\documentclass[11pt]{article}%
\usepackage{amsmath}
\usepackage{amsfonts}
\usepackage{amssymb}
\usepackage{amsthm}
\usepackage{graphicx}%
\providecommand{\U}[1]{\protect\rule{.1in}{.1in}}
\newtheorem{theorem}{Theorem}[section]

\newtheorem{lemma}[theorem]{Lemma}
\newtheorem{proposition}[theorem]{Proposition}
\newtheorem{remark}[theorem]{Remark}

\theoremstyle{definition}

\newtheorem{notation}[theorem]{Notation}

\theoremstyle{remark}
\newtheorem{example}[theorem]{Example}

\numberwithin{equation}{section}

\begin{document}

\title{A Neumann series of Bessel functions representation for solutions of perturbed Bessel equations}
\author{Vladislav V. Kravchenko$^1$, Sergii M. Torba$^1$ and Ra\'{u}l Castillo-P\'{e}rez$^2$\\{\small $^1$ Departamento de Matem\'{a}ticas, CINVESTAV del IPN, Unidad
Quer\'{e}taro, }\\{\small Libramiento Norponiente \#2000, Fracc. Real de Juriquilla,
Quer\'{e}taro, Qro., 76230 MEXICO.}\\
{\small $^2$ Maestr\'{\i}a en Telecomunicaciones, ESIME Zacatenco,  Instituto Polit\'{e}cnico Nacional,}\\{\small Av. Instituto Polit\'{e}cnico Nacional S/N, D.F. Mexico, 07738 MEXICO.}\\
{\small e-mail: vkravchenko@math.cinvestav.edu.mx,
storba@math.cinvestav.edu.mx, rcastillo@ipn.mx \thanks{Research was supported by CONACYT, Mexico
via the projects 166141 and 222478. R.~Castillo would like to thank the support of CONACYT and of the SIBE and EDI programs of the IPN as well as
that of the project SIP 20160525.}}}
\maketitle

\begin{abstract}
A new representation for a regular solution of the perturbed Bessel equation of the form $Lu=-u^{\prime\prime}+\left(  \frac{l(l+1)}{x^{2}}+q(x)\right)
u=\omega^{2}u$ is obtained. The solution is represented as a Neumann series of
Bessel functions uniformly convergent with respect to $\omega$. For the
coefficients of the series explicit direct formulas are obtained in terms of
the systems of recursive integrals arising in the spectral parameter power
series (SPPS) method, as well as convenient for numerical computation
recurrent integration formulas.

The result is based on application of several ideas from the classical
transmutation (transformation) operator theory, recently discovered mapping properties of the
transmutation operators involved and a Fourier-Legendre series expansion of
the transmutation kernel. For convergence rate estimates, asymptotic formulas, a Paley-Wiener theorem and some results from constructive approximation theory were used.

We show that the analytical representation obtained among other possible
applications offers a simple and efficient numerical method able to compute
large sets of eigendata with a nondeteriorating accuracy.

\end{abstract}

\section{Introduction}

In the present work the equation
\begin{equation}
-u^{\prime\prime}+\left(  \frac{l(l+1)}{x^{2}}+q(x)\right)  u=\omega
^{2}u,\qquad x\in(0,b], \label{I1}%
\end{equation}
is studied, where $l$ is a real number, $l\geq-\frac{1}{2}$, $q$ is a
complex-valued function on $[0,b]$ satisfying the following condition
\begin{equation}\label{Condition on q}
\begin{split}
    x q(x) &\in L_1(0,b) \qquad\text{if } l>-1/2, \\
    x^{1-\varepsilon} q(x) & \in L_1(0,b)\qquad \text{for some }\varepsilon>0 \text{ if } l=-1/2,
\end{split}
\end{equation}
and $\omega$ is a (complex)
spectral parameter. Denote $L=-\frac{d^{2}}{dx^{2}}+\frac{l(l+1)}{x^{2}}%
+q(x)$. Equations of the form (\ref{I1}) appear naturally in many real-world
applications after a separation of variables and therefore have received
considerable attention (see, e.g., \cite{BoumenirChanane}, \cite{CKT2013},
\cite{CKT2015}, \cite{Chebli1994}, \cite{Guillot 1988}, \cite{KosTesh2011},
\cite[Sect. 3.7]{Okamoto}, \cite{Weidmann}).

The main result of the work is a representation of a regular solution
$u_{l}(\omega,x)$ of (\ref{I1}) satisfying the asymptotic relation
$u_{l}(\omega,x)\sim x^{l+1}$ when $x\rightarrow0$ in the form of the
following series of Bessel functions%
\begin{equation}
u_{l}(\omega,x)=\frac{2^{l+1}\Gamma\left(  l+\frac{3}{2}\right)  }{\sqrt{\pi
}\omega^{l}}xj_{l}\left(  \omega x\right)  +\sum_{n=0}^{\infty}\left(
-1\right)  ^{n}\beta_{n}(x)j_{2n}(\omega x) \label{I2}%
\end{equation}
where $j_{k}$ denotes the spherical Bessel function of the first kind of order
$k$. For the coefficients $\beta_{n}$ explicit direct formulas are obtained in
terms of a system of recursive integrals arising in the SPPS method
\cite{CKT2013}. For a fixed $x$ the series in (\ref{I2}) represents a
so-called Neumann series of Bessel functions (see \cite{Watson},
\cite{Wilkins} and a recent publication on the subject \cite{Baricz et al} and
references therein).

We prove that the series (\ref{I2}) converges uniformly with respect to
$\omega$. More precisely, the very convenient estimates \eqref{est uln1} and \eqref{est uln2} are obtained
which guarantee that a partial sum from (\ref{I2}) approximates equally well
the solution $u_{l}(\omega,x)$ both for small and for large values of the
spectral parameter $\omega$. We illustrate this feature of (\ref{I2}) with
several numerical examples which show that this new representation besides
other possible applications can be used as a simple and powerful numerical
method for solving boundary value and spectral problems related to (\ref{I1}).

In the recent work \cite{KNT 2015} an analogous representation of solutions
was obtained for the regular one-dimensional Schr\"{o}dinger equation
$-y^{\prime\prime}+q(x)y=\omega^{2}y$. However following similar lines does
not lead to the result in the case of the perturbed Bessel equation. Several
new ideas are necessary. We make use of the properties of a couple of
transmutation operators, one of them relating the operators $\frac{d^{2}%
}{dx^{2}}$ and $\frac{d^{2}}{dx^{2}}-\frac{l(l+1)}{x^{2}}$ (studied in
\cite{Sitnik2010}, \cite{Santana Jessica}, \cite{KrST}) and the other relating the
operators $\frac{d^{2}}{dx^{2}}-\frac{l(l+1)}{x^{2}}$ and $\frac{d^{2}}%
{dx^{2}}-\frac{l(l+1)}{x^{2}}-q(x)$ (studied in \cite{Volk}, \cite{CoudrayCoz}). The first of
these transmutations is used for separating the part corresponding to the
unperturbed equation ($q\equiv0$) and the second to add the perturbation. We
specify that neither here nor in the title of the paper the perturbation means
any kind of smallness of the coefficient $q$. The use of two transmutation
operators allows us to show that the regular solution of (\ref{I1}) can be
represented in the form
\begin{equation}
u_{l}(\omega,x)=\frac{2^{l+1}\Gamma\left(  l+\frac{3}{2}\right)  }{\sqrt{\pi
}\omega^{l}}xj_{l}\left(  \omega x\right)  +\int_{0}^{x}R(x,t)\cos\omega
t\,dt\label{Intro ul}%
\end{equation}
where the kernel $R(x,t)$ is a sufficiently good function which admits a
convergent Fourier-Legendre series expansion. Moreover, this approach makes it
possible to write down the result of the integral $\int_{0}^{x}R(x,t)t^{2k}%
\,dt$ for any $k=0,1,2,\ldots$ and as a consequence to obtain explicit
formulas for the coefficients of the Fourier-Legendre series. Substitution of
the series  into (\ref{Intro ul}) leads to the main result (\ref{I2}). The convergence rate of the Fourier-Legendre series (and, consequently, of the series \eqref{I2}) depends on the smoothness of the integral kernel $R$. Only few basic properties of the kernel $R$ can be obtained using the results from \cite{Volk} and \cite{CoudrayCoz}. We implement a different approach based on the asymptotic formulas from \cite{KosSakhTesh2010} and \cite{FitouhiHamza}, a Paley-Wiener theorem and the constructive approximation theory \cite{DeVoreLorentz}. As a result, we present close to optimal convergence rate estimates depending on the parameter $l$ and the smoothness of the potential $q$.

Analogous formulas are developed for the derivative of the regular solution.

The direct explicit formula for the coefficient $\beta_{n}$ is not however the
most convenient for numerical computation. This is due to the fact that it
involves coefficients of the Legendre polynomial of order $2n$ which grow
rapidly when $n$ grows and hence, although the coefficients $\beta_{n}$
decrease, one needs to compute linear combinations of large numbers. This
reduces considerably the number of the coefficients which can be computed in
machine precision. Fortunately, there exists another way to compute the
coefficients $\beta_{n}$ with the aid of a recurrent integration procedure
similar to that arising in the SPPS method \cite{KrPorter2010}, \cite{KKRosu}
and much more stable in practice. To develop the procedure we find
the sequence of differential equations satisfied by the coefficients
$\beta_{n}$. This is done by substitution of (\ref{I2}) into (\ref{I1}).

In the last part of the paper we show that (\ref{I2}) offers a simple and
powerful numerical method for computing regular solutions of (\ref{I1}) and
for solving spectral problems related. The uniform convergence with respect to
$\omega$ allows one to compute large sets of eigendata with a
non-deteriorating accuracy.

\section{From regular to singular}

Consider the following integral operator defined on $C[0,b]$,%
\begin{equation*}
Y_{l}v(x):=\frac{x^{-l}}{2^{l+\frac{1}{2}}\Gamma\left(  l+\frac{3}{2}\right)
}\int_{0}^{x}\left(  x^{2}-s^{2}\right)  ^{l}v(s)\,ds.
\end{equation*}
The following statement is a slightly precised result from \cite{Sitnik2010}.

\begin{proposition}[\cite{Santana Jessica, KrST}] Let $v\in C^{2}[0,b]$ and $v^{\prime}(0)=0$. Then
\begin{equation*}
\left(  \frac{d^{2}}{dx^{2}}-\frac{l(l+1)}{x^{2}}\right)  Y_{l}v=Y_{l}%
\frac{d^{2}}{dx^{2}}v.
\end{equation*}
\end{proposition}

In particular \cite{KrST},
\begin{equation}
Y_{l}:\,x^{2k}\mapsto\frac{\Gamma\left(  k+\frac{1}{2}\right)  \Gamma\left(
l+1\right)  }{2^{l+\frac{3}{2}}\Gamma\left(  l+\frac{3}{2}\right)
\Gamma\left(  k+l+\frac{3}{2}\right)  }x^{2k+l+1},\quad k=0,1,2,\ldots.
\label{Ylx2k}%
\end{equation}

Denote%
\[
b_{l}(\omega x):=\sqrt{\omega x}J_{l+\frac{1}{2}}\left(  \omega x\right)  .
\]
This function is a regular solution of the equation
\[
\left(  \frac{d^{2}}{dx^{2}}-\frac{l(l+1)}{x^{2}}\right)  v=-\omega
^{2}v,\qquad x\in(0,b].
\]
Its corresponding power series has the form%
\[
b_{l}(\omega x)=(\omega x)^{l+1}\sum_{k=0}^{\infty}\frac{(-1)^{k}(\omega
x)^{2k}}{2^{2k+l+\frac{1}{2}}\Gamma\left(  k+1\right)  \Gamma\left(
k+l+\frac{3}{2}\right)  }.
\]

\begin{remark}
\label{Rem Yl[cos]}From \eqref{Ylx2k} we obtain
\[
Y_{l}\left[  \cos\omega x\right]  =\frac{\sqrt{\pi}\Gamma\left(  l+1\right)
}{2\omega^{l+1}\Gamma\left(  l+\frac{3}{2}\right)  }b_{l}(\omega x).
\]

\end{remark}

\section{Transmutation of Bessel-type operators}
\label{Sect Transmut Bessel}
Throughout this section we assume that $q\in C[0,b]$.
In \cite{Volk} the existence of a unique continuous kernel $V(x,t)$ was proved
such that for all $\omega\in\mathbb{C}$ the function
\[
u(\omega,x)=\mathcal{T}\left[  b_{l}(\omega x)\right]  :=b_{l}(\omega
x)+\int_{0}^{x}V(x,t)b_{l}(\omega t)\,dt
\]
is a regular solution of the equation
\begin{equation}
\left(  \frac{d^{2}}{dx^{2}}-\frac{l(l+1)}{x^{2}}-q(x)\right)  u=-\omega
^{2}u,\qquad x\in(0,b] \label{Bessel type}%
\end{equation}
and
\begin{equation}
V(x,x)=\frac{Q(x)}{2} \label{V(x,x)}%
\end{equation}
where $Q(x):=\int_{0}^{x}q(t)dt$, see also \cite{CoudrayCoz}.

If $l\neq0$ then the left endpoint is singular. Despite that, the equation
$Lu=0$ possesses a solution $\phi(x)$ which is bounded at $x=0$ and satisfies
the following asymptotics at $x=0$
\begin{align}
\phi(x)  &  \sim x^{l+1},\quad x\rightarrow0,\label{SolAsymptotic}\\
\phi^{\prime}(x)  &  \sim(l+1)x^{l},\quad x\rightarrow0,
\label{DerSolAsymptotic}%
\end{align}
see, e.g., \cite[Lemma 3.2]{KosTesh2011} for a real-valued $q$. In
\cite{CKT2013} an explicit construction of the solution with this asymptotics
at zero for the general case of a complex-valued $q$ was given.

From now on we assume that there exists a non-vanishing on $(0,b]$
complex-valued solution $u_{0}$ of the equation
\begin{equation}
-u_{0}^{\prime\prime}+\left(  \frac{l(l+1)}{x^{2}}+q(x)\right)  u_{0}=0
\label{PartSolEq}%
\end{equation}
satisfying together with its first derivative the asymptotic relations
\eqref{SolAsymptotic} and \eqref{DerSolAsymptotic}. In \cite{CKT2013} the
existence and a procedure for construction of such a solution was given in the
case when $q(x)\geq0$, $x\in(0,b]$.

\begin{notation}\label{Not phik}
Let us define the following system of functions
\[
\varphi_{n}(x):=(-1)^{n}(2n)!u_{0}(x)\widetilde{X}^{(2n)}(x)
\]
where
\begin{equation}%
\begin{split}
\widetilde{X}^{(0)}  &  \equiv1,\\
\widetilde{X}^{(n)}(x)  &  =%
\begin{cases}
\displaystyle\int_{0}^{x}u_{0}^{2}(t)\widetilde{X}^{(n-1)}(t)\,dt, &
\text{if }n\text{ is odd},\\
-\displaystyle\int_{0}^{x}\frac{\widetilde{X}^{(n-1)}(t)}{u_{0}^{2}(t)}\,dt, &
\text{if }n\text{ is even}.
\end{cases}
\end{split}
\label{Xtilde}%
\end{equation}
We keep the notation $\widetilde{X}$ for consistency with
other publications on the SPPS method, see, e.g., \cite{KrPorter2010},
\cite{KKRosu}, \cite{KT Obzor}.
\end{notation}

In \cite{CKT2013} it was proved that
\begin{equation}
\mathcal{T}\left[  x^{2k+l+1}\right]  =(-1)^{k}2^{2k}k!\left(  l+\frac{3}%
{2}\right)  _{k}u_{0}(x)\widetilde{X}^{(2k)}(x)\quad\text{for any
}k=0,1,2,\ldots. \label{mapping property}%
\end{equation}

\begin{theorem}
\label{Th u(omega,x)}Let $q\in C[0,b]$, $l\geq-\frac{1}{2}$. There exists a
continuous function $R^{(0)}(x,t)$, $0\leq t\leq x\leq b$ such that for any
$\omega\in\mathbb{C}$ the function
\begin{equation}
u(\omega,x)=a(\omega)b_{l}(\omega x)+\int_{0}^{x}R^{(0)}(x,t)\cos\omega t\,dt
\label{uTh}%
\end{equation}
with
\[
a(\omega)=\frac{\sqrt{\pi}\Gamma\left(  l+1\right)  }{2\omega^{l+1}%
\Gamma\left(  l+\frac{3}{2}\right)  }%
\]
is a regular solution of \eqref{Bessel type}.
\end{theorem}

\begin{proof}
Consider
\begin{equation}
u(\omega,x)=\mathcal{T}Y_{l}\left[  \cos\omega x\right]  =a(\omega
)b_{l}(\omega x)+\int_{0}^{x}V(x,t)Y_{l}\left[  \cos\omega t\right]  \,dt.
\label{u(omega x)}%
\end{equation}
This function is a solution of (\ref{Bessel type}) due to the fact that
$Y_{l}\left[  \cos\omega t\right]  =a(\omega)b_{l}(\omega x)$.
Hence it is sufficient to prove that (\ref{u(omega x)}) can be written in the
form (\ref{uTh}).

Consider
\begin{align*}
\int_{0}^{x}V(x,t)Y_{l}\left[  \cos\omega t\right]  \,dt  &  =\frac
{1}{2^{l+\frac{1}{2}}\Gamma\left(  l+\frac{3}{2}\right)  }\int_{0}%
^{x}V(x,t)t^{-l}\int_{0}^{t}\left(  t^{2}-s^{2}\right)  ^{l}\cos\omega
s\,dsdt\\
&  =\frac{1}{2^{l+\frac{1}{2}}\Gamma\left(  l+\frac{3}{2}\right)  }\int
_{0}^{x}\cos\omega s\int_{s}^{x}V(x,t)t^{-l}\left(  t^{2}-s^{2}\right)
^{l}\,dtds.
\end{align*}
Denote
\begin{equation}\label{R(0)}
R^{(0)}(x,s)    =\frac{1}{2^{l+\frac{1}{2}}\Gamma\left(  l+\frac{3}%
{2}\right)  }\int_{s}^{x}V(x,t)t^{-l}\left(  t^{2}-s^{2}\right)
^{l}\,dt
  =\frac{1}{2^{l+\frac{1}{2}}\Gamma\left(  l+\frac{3}{2}\right)  }\int
_{s}^{x}V(x,t)\left(  t-\frac{s^{2}}{t}\right)  ^{l}\,dt. %
\end{equation}
Since $V$ is continuous, $R^{(0)}$ is continuous as well.
\end{proof}

The solution (\ref{uTh}) can be written in the form of an SPPS \cite{CKT2013},
\begin{equation}
u(\omega,x)=\frac{\sqrt{\pi}\Gamma\left(  l+1\right)  }{2^{l+\frac{3}{2}%
}\Gamma^{2}\left(  l+\frac{3}{2}\right)  }u_{0}(x)\sum_{k=0}^{\infty}%
\omega^{2k}\widetilde{X}^{(2k)}(x). \label{u via SPPS}%
\end{equation}
Indeed, we have that
\begin{align*}
u(\omega,x)  &  =\frac{a(\omega)}{2^{l+\frac{1}{2}}}\sum_{k=0}^{\infty}\frac{\left(
-1\right)  ^{k}\omega^{2k+l+1}}{2^{2k}\Gamma\left(  k+1\right)  \Gamma\left(
k+l+\frac{3}{2}\right)  }\mathcal{T}\left[  x^{2k+l+1}\right] \\
& =\frac{a(\omega)}{2^{l+\frac{1}{2}}}\sum_{k=0}^{\infty}\frac{\omega
^{2k+l+1}\left(  l+\frac{3}{2}\right)  _{k}u_{0}(x)\widetilde{X}^{(2k)}%
(x)}{\Gamma\left(  k+l+\frac{3}{2}\right)  }
\end{align*}
where we used (\ref{mapping property}). Taking into account the definition of
$a(\omega)$ and the identity $\left(  l+\frac{3}{2}\right)  _{k}=\Gamma\left(
k+l+\frac{3}{2}\right)  /\Gamma\left(  l+\frac{3}{2}\right)  $ we obtain
(\ref{u via SPPS}). Thus,
\begin{align*}
u_{0}(x)\sum_{k=0}^{\infty}\omega^{2k}\widetilde{X}^{(2k)}(x)  &
=\frac{2^{l+\frac{3}{2}}\Gamma^{2}\left(  l+\frac{3}{2}\right)  }{\sqrt{\pi
}\Gamma\left(  l+1\right)  }a(\omega)b_{l}(\omega x)+\int_{0}^{x}%
R(x,t)\cos\omega t\,dt\\
&  =\frac{2^{l+\frac{1}{2}}\Gamma\left(  l+\frac{3}{2}\right)  }{\omega^{l+1}%
}b_{l}(\omega x)+\int_{0}^{x}R(x,t)\sum_{k=0}^{\infty}\frac{\left(  -1\right)
^{k}\left(  \omega t\right)  ^{2k}}{\left(  2k\right)  !}\,dt
\end{align*}
where
\begin{equation}
R(x,t):=\frac{2^{l+\frac{3}{2}}\Gamma^{2}\left(  l+\frac{3}{2}\right)  }%
{\sqrt{\pi}\Gamma\left(  l+1\right)  }R^{(0)}(x,t). \label{R=Const R0}%
\end{equation}
Hence,
\[
u_{0}(x)\sum_{k=0}^{\infty}\omega^{2k}\widetilde{X}^{(2k)}(x)=\Gamma\left(
l+\frac{3}{2}\right)  \sum_{k=0}^{\infty}\frac{\left(  -1\right)  ^{k}%
\omega^{2k}x^{2k+l+1}}{2^{2k}\Gamma\left(  k+1\right)  \Gamma\left(
k+l+\frac{3}{2}\right)  }+\sum_{k=0}^{\infty}\frac{\left(  -1\right)
^{k}\omega^{2k}}{\left(  2k\right)  !}\int_{0}^{x}R(x,t)t^{2k}\,dt.
\]
The coefficients of the equal uniformly convergent power series with respect
to $\omega$ must coincide, thus,
\[
u_{0}(x)\widetilde{X}^{(2k)}(x)=\frac{\left(  -1\right)  ^{k}\Gamma\left(
l+\frac{3}{2}\right)  x^{2k+l+1}}{2^{2k}\Gamma\left(  k+1\right)
\Gamma\left(  k+l+\frac{3}{2}\right)  }+\frac{\left(  -1\right)  ^{k}}{\left(
2k\right)  !}\int_{0}^{x}R(x,t)t^{2k}\,dt.
\]
Hence
\begin{equation}
\varphi_{k}(x)=c_{k,l}x^{2k+l+1}+\int_{0}^{x}R(x,t)t^{2k}\,dt\quad\text{for
any }k=0,1,2,\ldots\label{phik=int}%
\end{equation}
where
\begin{equation}\label{ckl}
c_{k,l}:=\frac{\Gamma\left(  l+\frac{3}{2}\right)  \Gamma\left(
k+\frac{1}{2}\right)  }{\sqrt{\pi}\Gamma\left(  k+l+\frac{3}{2}\right)  },
\end{equation}
and we used the identities $\left(  2k\right)  !=\Gamma\left(  2k+1\right)
=k2^{2k}\Gamma\left(  k\right)  \Gamma\left(  k+\frac{1}{2}\right)  /\sqrt
{\pi}$.

\section{A Fourier-Legendre representation of the kernel $R(x,t)$}

Multiplication of equality (\ref{uTh}) by the constant $\frac{2^{l+\frac{3}%
{2}}\Gamma^{2}\left(  l+\frac{3}{2}\right)  }{\sqrt{\pi}\Gamma\left(
l+1\right)  }$ from (\ref{R=Const R0}) allows us to write down a regular
solution of (\ref{Bessel type}) in the form%
\begin{equation}
u_{l}(\omega,x)=d(\omega)b_{l}(\omega x)+\int_{0}^{x}R(x,t)\cos\omega t\,dt
\label{ul}%
\end{equation}
with $d(\omega):=\frac{2^{l+\frac{1}{2}}\Gamma\left(  l+\frac{3}{2}\right)
}{\omega^{l+1}}$. It is related with the solution from Theorem
\ref{Th u(omega,x)} by $u_{l}(\omega,x)=\frac{2^{l+\frac{3}{2}}\Gamma
^{2}\left(  l+\frac{3}{2}\right)  }{\sqrt{\pi}\Gamma\left(  l+1\right)
}u(\omega,x)$ and for any $\omega$ satisfies the following asymptotic equality
when $x\rightarrow0$, $u_{l}(\omega,x)\sim x^{l+1}$. As a next step we
construct a Fourier-Legendre representation for the kernel $R(x,t)$.

First, we
need to recall the following notations, c.f. \cite[Chap.\ 2, \S7 and \S9]{DeVoreLorentz}. For $\alpha>0$ we write $\alpha=r+\beta$, where $r\in\mathbb{Z}$ and $0<\beta\le 1$, and say that a function $f$ belongs to $\operatorname{Lip}_\alpha(I)$ class, with $I$ being either a segment or the whole line, if $f\in C^{r}(I)$ and $f^{(r)}\in \operatorname{Lip}_\beta(I)$. Consider the difference operator $\Delta_h:L_p(I)\to L_p(I_h)$ acting on a function $f$ as $\Delta_h f(\cdot) = f(\cdot+h) - f(\cdot)$, here $I_h:=[a,b-h]$ if $I=[a,b]$, $h<b-a$ and $I_h:=I$ if $I=\mathbb{R}$. Then the $r$-th modulus of smoothness of $f$ is defined by
\[
\omega_r(f,t)_{L_p(I)}:=\sup_{0<h\le t}\| \Delta_h^r(f)\|_{L_p(I_{rh})}.
\]
For $\alpha>0$ let $r$ be the smallest integer satisfying $r>\alpha$, i.e., $r=[\alpha]+1$. Then the generalized Lipschitz class $\operatorname{Lip}_\alpha^*(I,p)$ is defined as the class of functions $f\in L_p(I)$ satisfying $\omega_r(f,t)_{L_p(I)}\le M t^\alpha$ for all $t>0$ with some constant $M=M(f)$. By $W_2^\alpha(\mathbb{R})$, $\alpha\ge 0$ we denote the fractional-order Sobolev space, also called Bessel potential space \cite[Chap. 7]{Adams} consisting of the functions satisfying $f\in L_2(\mathbb{R})$ and $(1+|\xi|^2)^{\alpha/2}\mathcal{F}[f](\xi)\in L_2(\mathbb{R})$, where $\mathcal{F}$ is the Fourier transform operator.

Following \cite{KosSakhTesh2010} we introduce the notation
\[
\tilde q(x) = \begin{cases}
|q(x)|, & l>-1/2,\\
\bigl(1-\log(x/b)\bigr)|q(x)|,& l=-1/2.
\end{cases}
\]

\begin{proposition}\label{Prop Smooth R}
Let $q$ satisfy the condition \eqref{Condition on q}. Suppose additionally that
\begin{equation}\label{Condition precise on q}
x^\alpha \tilde q(x)\in L_1(0,b)\qquad \text{for some }\alpha\in[0,1],\ \alpha<3/2+l.
\end{equation}
Let $x>0$ be fixed. Then there exists an even, compactly supported on $[-x,x]$ function $\widetilde R(x,t)$ such that
\begin{enumerate}
\item $\widetilde R\in W_2^{l+3/2-\alpha-\varepsilon}(\mathbb{R})$ for any sufficiently small $\varepsilon>0$;
if $\alpha < l+1$ then additionally
$\widetilde R\in \operatorname{Lip}_{l+1-\alpha-\varepsilon}(\mathbb{R})$.
\item $\widetilde R\in \operatorname{Lip}^*_{l+3/2-\alpha}(\mathbb{R},2)$;
\item the function $R$ from \eqref{ul} satisfies
\[
R(x,t) = 2 \widetilde R(x,t),\qquad 0\le t\le x.
\]
\end{enumerate}
\end{proposition}

\begin{proof}
Consider the function
\[
g(\omega):= u_l(\omega, x) - d(\omega)b_l(\omega x).
\]
In \cite[Lemma 2.18]{KosSakhTesh2010} it was proved under the condition
$\int_0^b y\tilde q(y)\,dy <\infty$
(satisfied automatically whenever \eqref{Condition on q} holds)
that $g(\omega)$ is an entire function and for all $\omega \in\mathbb{C}$ satisfies the following estimate
\begin{equation}\label{EstForg}
    |g(\omega)|\le C\left(\frac{x}{b+|\omega|x}\right)^{l+1} e^{|\operatorname{Im}\omega| x} \int_0^x \frac{y\tilde q(y)}{b+|\omega|y} \,dy,
\end{equation}
where $C= C_l^2 \exp\left( C_l\int_0^b y\tilde q(y)\,dy\right)$ and the constant $C_l$ does not depend on $q$ and $x$.

Since $\frac t{b+|\omega|t} \le \frac 1{|\omega|}$, it follows from \eqref{Condition precise on q} and \eqref{EstForg} that
\begin{equation}\label{EstForgR}
    |g(\omega)|\le \frac C{b^\alpha|\omega|^{l+2-\alpha}}\int_0^x y^\alpha\tilde q(y)\,dy\le \frac {\tilde C}{|\omega|^{l+2-\alpha}},\qquad \omega\in\mathbb{R},
\end{equation}
showing that $g\in L^2(\mathbb{R})$. Applying the Paley-Wiener theorem \cite[Thm. VI.7.4]{Katznelson} we obtain that the Fourier transform of the function $g$ (which we denote by $\widetilde R$) is compactly supported on $[-x,x]$, i.e.,
\begin{equation}\label{g omega R}
g(\omega)= \int_{-x}^x \widetilde R(x,t) e^{i\omega t}\,dt.
\end{equation}
Note additionally that both functions $u_l(\omega,x)$ and $d(\omega)b_l(\omega x)$ are even functions of the real variable $\omega$, hence $\widetilde R$ is also even and
\begin{equation}\label{ulR}
    u_l(\omega, x) - d(\omega)b_l(\omega x) = g(\omega) = 2\int_0^x \widetilde R(x,t)\cos \omega t \,dt.
\end{equation}
Since the equalities \eqref{ul} and \eqref{ulR} hold for all $\omega$, we  conclude that $2\widetilde R = R$ a.e.\ for $0\le t\le x$.

The inclusion $\widetilde R\in W_2^{l+3/2-\alpha-\varepsilon}(\mathbb{R})$ follows from \eqref{EstForgR} and from the definition of Bessel potential spaces via the Fourier transform. The inclusion  $\widetilde R\in \operatorname{Lip}_{l+1-\alpha-\varepsilon}(\mathbb{R})$ follows from the embedding theorem $W_2^{\beta+1/2}(\mathbb{R})\subset \operatorname{Lip}_{\beta}(\mathbb{R})$ valid for any $\beta>0$, $\beta\not\in\mathbb{N}$,  see, e.g., \cite[Sect. 2.8.1]{Triebel1995}.

For the last statement of the proposition we use the following generalization of \cite[Theorem 85]{Titchmarsh}. Let $f$ belong to $L^2(\mathbb{R})$ and its Fourier transform $F$ satisfies $(\int_{-\infty}^{-X}+\int_X^\infty) |F(x)|^2dx \le C^2 X^{-2\beta}$ for some $\beta>0$ and all $X>0$. Then
\begin{equation}\label{omegar Estimate}
    \omega_r(f,h)_{L_2(\mathbb{R})}\le C_r h^\beta,\qquad r=[\beta]+1 \quad\text{and}\quad C_r = C\sqrt{2+\frac{r}{r-\beta}}.
\end{equation}
We omit the proof of this fact since it is similar to that of \cite{Titchmarsh} with the only difference that the equality $\int_{-\infty}^\infty |\Delta_h^r f(x)|^2\,dx = \int_{-\infty}^\infty \sin^{2r} xh\cdot |F(x)|^2\,dx$ is used. Inequality \eqref{EstForgR} implies that $(\int_{-\infty}^{-X}+\int_X^\infty) |g(\omega)|^2d\omega \le \frac{2\tilde C^2}{2l+3-2\alpha} X^{-(2l+3-2\alpha)}$ proving the inclusion
$\widetilde R\in \operatorname{Lip}^*_{l+3/2-\alpha}(\mathbb{R},2)$.
\end{proof}

\begin{remark}
It is possible to obtain the smoothness properties of the integral kernel $R$ directly from \eqref{R(0)} and \eqref{R=Const R0}. In particular, one may verify by somewhat lengthy calculations that $\widetilde R$, the continuation of $R$ onto $\mathbb{R}$ as an even compactly supported function of $t$, belongs to $\operatorname{Lip}_{1+l}(\mathbb{R})$, a slight improvement as compared to Proposition \ref{Prop Smooth R}.
Note that compared to \eqref{Condition on q} the condition \eqref{Condition precise on q} does not imply additional restrictions on $q$, it only specifies the order of the singularity at zero (if any).
\end{remark}

Let $P_{n}$ denote the Legendre polynomial of order $n$, $l_{k,n}$ be the
corresponding coefficient of $x^{k}$, that is $P_{n}(x)=\sum_{k=0}^{n}%
l_{k,n}x^{k}$.

\begin{theorem}
\label{Th Legendre for R}Let $q$ satisfy \eqref{Condition on q}. Then the kernel $R(x,t)$ has the form
\begin{equation}
R(x,t)=\sum_{n=0}^{\infty}\frac{\beta_{n}(x)}{x}P_{2n}\left(  \frac{t}%
{x}\right)  \label{R(x,t)}%
\end{equation}
with $\beta_{n}$ being defined by the equality
\begin{equation}
\beta_{n}(x)=\left(  4n+1\right)  \sum_{k=0}^{n}\frac{l_{2k,2n}}{x^{2k}%
}\left(  \varphi_{k}(x)-c_{k,l}x^{2k+l+1}\right), \label{beta n}
\end{equation}
and $c_{k,l}$ being given by \eqref{ckl}. For any $l\ge -1/2$, the series in \eqref{R(x,t)} converges in the $L_2$ norm.

Let additionally $q$ satisfy \eqref{Condition precise on q}. If $l>\alpha-1/2$ then for any $x\in(0,b]$ the series in \eqref{R(x,t)} converges uniformly with respect to $t\in\left[  0,x\right]  $; if $\alpha-1<l \le \alpha-1/2$,  $l\ge-1/2$, then for any $x\in (0,b]$ the series converges uniformly with respect to $t\in [0,x']\subset\left[  0,x\right)$.

Let
\begin{equation}\label{RN(x,t)}
    R_N(x,t):= \sum_{n=0}^N \frac{\beta_n(x)}{x} P_{2n}\left(\frac tx\right).
\end{equation}
There exist constants $C_1$ and $C_2$, dependent on $q$ and $l$ and independent of $x$ and $N$, such that for any $x>0$
\begin{equation}\label{L2 estimate R RN}
    \|R(x,\cdot) - R_N(x,\cdot)\|_{L_2[0,x]} \le \frac{C_1 x^{l+3/2-\alpha}}{N^{l+3/2-\alpha}},\qquad 2N\ge [l+5/2]
\end{equation}
and
\begin{equation}\label{betan est}
    |\beta_N(x)|\le \frac{C_2 x^{l+2-\alpha}}{(N-1)^{l+1-\alpha}},\qquad 2N\ge [l+9/2].
\end{equation}
\end{theorem}

\begin{proof}
For any $x\in(0,b]$ the kernel $R(x,\cdot)\in L_2[0,x]$. Hence it admits a Fourier-Legendre series representation of the form
$\sum_{j=0}^{\infty}A_{j}(x)P_{2j}\left(  \frac{t}{x}\right)  $. For
convenience we consider $A_{j}(x)=\frac{\beta_{j}(x)}{x}$. Note that
\begin{equation*}
\int_{0}^{x}R(x,t)P_{2n}\left(  \frac{t}{x}\right)  \,dt    =\sum
_{j=0}^{\infty}\frac{\beta_{j}(x)}{x}\int_{0}^{x}P_{2j}\left(  \frac{t}%
{x}\right)  P_{2n}\left(  \frac{t}{x}\right)  \,dt  =\frac{\beta_{n}(x)}{4n+1}.
\end{equation*}
Hence $\beta_{n}(x)=\left(  4n+1\right)  \int_{0}^{x}R(x,t)P_{2n}\left(
\frac{t}{x}\right)  \,dt$. On the other hand we have%
\begin{equation*}
\int_{0}^{x}R(x,t)P_{2n}\left(  \frac{t}{x}\right)  \,dt    =\sum_{k=0}%
^{n}\frac{l_{2k,2n}}{x^{2k}}\int_{0}^{x}R(x,t)t^{2k}\,dt
  =\sum_{k=0}^{n}\frac{l_{2k,2n}}{x^{2k}}\left(  \varphi_{k}(x)-c_{k,l}%
x^{2k+l+1}\right)
\end{equation*}
where (\ref{phik=int}) was used. Thus, (\ref{beta n}) is obtained. Note that $q$ does not need to be continuous on $[0,b]$ for the equality \eqref{phik=int} to hold, the condition \eqref{Condition on q} is sufficient. Indeed, the functions $\varphi_k$ are defined by the same formulas \eqref{Xtilde} (their validity under the condition \eqref{Condition on q} can be verified similarly to \cite{BCK2015}), the SPPS representation \eqref{u via SPPS} and the integral representation \eqref{ul} hold, and the proof from Section \ref{Sect Transmut Bessel} can be easily repeated.

Now let additionally $q$ satisfy \eqref{Condition precise on q}.
Consider the restriction of the function $\widetilde R$ from Proposition \ref{Prop Smooth R} to the segment $[-x,x]$. Since it is an even function, its Fourier-Legendre series contains only even terms and due to the equality $R(x,t) = 2\widetilde R(x,t)$, $0\le t\le x$ one has $\widetilde R(x,t) = \sum_{n=0}^\infty \frac{\beta_n(x)}{2x}P_{2n}\left(\frac tx\right)$, where the series converges in $L_2[-x,x]$.

Theorem 4.10 from \cite{Suetin} states that if a function $g\in \operatorname{Lip}_{\beta}[-1,1]$, where $\beta > 1/2$, then the partial sums of the Fourier-Legendre series of $g$ converge uniformly on $[-1,1]$ to the function $g$. By Proposition \ref{Prop Smooth R}, $\widetilde R\in \operatorname{Lip}_{1+l-\alpha-\varepsilon}(\mathbb{R})$, hence its restriction onto $[-x,x]$ belongs to $\operatorname{Lip}_{1+l-\alpha-\varepsilon}[-x,x]$, which is sufficient to establish the uniform convergence of the series \eqref{R(x,t)} for any $l>\alpha-1/2$. For $l>\alpha-1$, $l\ge -1/2$, \cite[Corollary to Theorem XIII]{Jackson} asserts the uniform convergence of the Fourier-Legendre series of the function $\widetilde R$ on any $[-x+\varepsilon,x-\varepsilon]\subset(-x,x)$, i.e., the series \eqref{R(x,t)} converges uniformly with respect to $t\in[0,x-\varepsilon]\subset[0,x)$ for any $\varepsilon>0$.

Consider the functions $g(z):=2\widetilde R(x,xz)$ and $g_N(z):=R_N(x,xz)$, $z\in[-1,1]$. The function $g_N$ is a polynomial of degree $2N$ and is a partial sum of the Fourier-Legendre series of $g$, i.e., $g_N$ coincides with the polynomial of the best $L_2[-1,1]$ approximation of the function $g$ by polynomials of degree $2N$. Hence by Theorem 6.3 from \cite[Chap. 7]{DeVoreLorentz} for any $r\in\mathbb{N}$ there exists a universal constant $C_r$ such that $\|g-g_N\|_{L_2[-1,1]}\le C_r\omega_r(g, \frac 1{2N})_{L_2[-1,1]}$, $2N\ge r$. We take $r=[l+3/2-\alpha]+1$. Then using the estimates \eqref{EstForgR} and \eqref{omegar Estimate} we obtain  that
\begin{equation*}
    \omega_r\left(g,\frac 1{2N}\right)_{L_2[-1,1]} =
    \frac 1{\sqrt x} \omega_r\left(\widetilde R,\frac x{2N}\right)_{L_2[-x,x]}\le
    \frac 1{\sqrt x} \omega_r\left(\widetilde R,\frac x{2N}\right)_{L_2(\mathbb{R})}\le \frac{C(q)}{\sqrt{x}}\left(\frac{x}{2N}\right)^{l+3/2-\alpha},
\end{equation*}
where the constant $C(q)$ depends neither on $x$ nor on $N$. To finish the proof of \eqref{L2 estimate R RN}, note that $\|R(x,\cdot)-R_N(x,\cdot)\|_{L_2[0,x]} = \frac 12\|2\widetilde R(x,\cdot)-R_N(x,\cdot)\|_{L_2[-x,x]} = \frac 12\sqrt{x}\|g-g_N\|_{L_2[-1,1]}$.

To prove the estimate \eqref{betan est} we proceed as follows.
\begin{equation*}
    \begin{split}
       |\beta_N(x)| &= (4N+1) \left|\int_0^x R(x,t) P_{2N}\left(\frac tx\right)\,dt\right|  \\
       &=  (4N+1) \left|\int_0^x \bigl(R(x,t)-R_{N-1}(x,t)\bigr) P_{2N}\left(\frac tx\right)\,dt\right| \\
         & \le (4N+1) \| R(x,\cdot) - R_{N-1}(x,\cdot)\|_{L_2[0,x]}\cdot \sqrt{\frac{x}{4N+1}}\le \frac {C_1\sqrt{4N+1} \cdot x^{l+2-\alpha}}{(N-1)^{l+3/2-\alpha}}\le \frac{C_2 x^{l+2-\alpha}}{(N-1)^{l+1-\alpha}},
     \end{split}
\end{equation*}
where we used the Cauchy-Schwarz inequality and the fact that $R_{N-1}$ is a polynomial in even powers of $t$ of degree lower than $2N$, hence orthogonal to $P_{2N}$.
\end{proof}

Note that the estimates \eqref{L2 estimate R RN} and \eqref{betan est} do not depend on the smoothness of the potential $q$. In Subsection \ref{Subsect decay analysis} we study the behavior of the coefficients $\beta_n$ numerically
and observe that for some potentials the actual degree of $N$ in the decay rate of the coefficients $\beta_n$ is $2l+3$, higher than $l+1-\alpha$ in \eqref{betan est}. Below we give a proof for such improved decay rate requiring $q$ to be sufficiently smooth.
Additionally, in Subsection \ref{Subsect decay analysis} we observe  that in the special case $l\in\mathbb{N}_0:=\mathbb{N}\cup\{0\}$ the coefficients $\beta_n$ decay much faster than for $l\not\in\mathbb{N}_0$ and that the decay rate depends on the smoothness of the potential $q$. Below we present a theoretical justification of this phenomenon. We need the following lemma first.

\begin{lemma}\label{Lemma beta k decay special}
Let $l\not\in\mathbb{N}_0$ and $k\in\mathbb{N}_0$. Then there exists a constant $c_k$ such that the following inequalities hold
\begin{equation}\label{beta k decay special}
\biggl|\int_0^x \left(1-\frac{t^2}{x^2}\right) ^{l+k}P_{2n}\left(\frac tx\right) \,dt\biggr| \le \frac{c_kx}{n^{2l+2k+2}},\qquad n\ge l+k+2.
\end{equation}
\end{lemma}
\begin{proof}
Using the Taylor series for the function $\left(1-\frac{t^2}{x^2}\right) ^{l+k}$ and the formula \cite[2.17.1]{Prudnikov} we obtain that
\[
\begin{split}
\int_0^x \left(1-\frac{t^2}{x^2}\right) ^{l+k}P_{2n}\left(\frac tx\right) \,dt &=  \sum_{m=n}^\infty (-1)^m \binom{l+k}{m} \frac 1{x^{2m}} \int_0^x t^{2m}P_{2n}\left(\frac tx\right) \,dt \\
&= \frac{x}{2 \Gamma(-l-k)}\sum_{m=n}^\infty \frac{\Gamma(m-l-k)\Gamma(m+1/2)}{\Gamma(m-n+1)\Gamma(m+n+3/2)}.
\end{split}
\]
Denote the terms  of the last series as $a_m$. We have that $a_m\ge 0$, and one can check by a simple verification that $\frac{a_{m+1}}{a_m}\ge 1$ for $m\le\frac{n^2-3/2}{l+k+2}$ and $\frac{a_{m+1}}{a_m}<1$ for $m>\frac{n^2-3/2}{l+k+2}\ge n-1$. Using the asymptotic formula \cite[(6.1.40)]{Abramowitz}, $\log \Gamma(z) = \bigl(z-\frac 12\bigr)\log z - z+\frac 12 \log 2\pi + O\bigl(\frac 1z\bigr)$, we can check that
\[
\log a_m = -(l+k+2)\log m + O\left(\frac 1 m\right),\qquad m\ge n,
\]
with the uniform bound of the error term. Hence $a_m\le \frac{c}{m^{k+l+2}}$, $m\ge n$ and
\[
\begin{split}
\sum_{m=n}^\infty a_m &\le \sum_{m=n}^{\bigl[\frac{n^2-3/2}{l+k+2}+1\bigr]} a_{\left[\frac{n^2-3/2}{l+k+2}+1\right]} + \sum_{m=\left[\frac{n^2-3/2}{l+k+2}+2\right]}^\infty \frac c{m^{l+k+2}}\\
&\le \frac{n^2}{l+k+2}\frac{c(l+k+2)^{l+k+2}}{(n^2-3/2)^{l+k+2}} + \frac{c}{l+k+1}\frac{(l+k+2)^{l+k+1}}{(n^2-3/2)^{l+k+1}},
\end{split}
\]
finishing the proof.
\end{proof}

\begin{proposition}\label{Prop Int l}
Suppose that $l\in \mathbb{N}$ and $q\in C^{2p-1}[0,b]$ for some $p\in\mathbb{N}$. Then the function $\widetilde R$ from Proposition \ref{Prop Smooth R} satisfies $\widetilde R\in \operatorname{Lip}_{l+1+p-\varepsilon}([-x,x])$ for any $\varepsilon>0$ and
$\widetilde R\in \operatorname{Lip}^*_{l+p+3/2}([-x,x],2)$. Moreover, there exist constants $c_1$ and $c_2$, dependent on $q$, $l$ and $p$, such that for any $x>0$ the inequalities hold
\begin{equation}\label{L2 estimate R RN Int l}
    \| R(x,\cdot) - R_N(x,\cdot)\|_{L_2[0,x]}\le \frac{c_1 x^{l+p+3/2}}{N^{l+p+3/2}},\qquad 2N\ge [l+p+5/2]
\end{equation}
and
\begin{equation}\label{betan estimate Int l}
    |\beta_N(x)|\le \frac{c_2 x^{l+p+2}}{(N-1)^{l+p+1}},\qquad 2N\ge [l+p+9/2].
\end{equation}

Suppose that $l\not\in\mathbb{N}_0$ and $q\in C^{2p-1}[0,b]$. Then there exist constants $c_3$ and $c_4$ such that for any $x>0$ the inequalities hold
\begin{equation}\label{L2 estimate R RN NonInt l}
    \| R(x,\cdot) - R_N(x,\cdot)\|_{L_2[0,x]}\le \frac{c_3 x^{l+3/2}}{N^r},\qquad N\ge [l+p+3]
\end{equation}
and
\begin{equation}\label{betan estimate NonInt l}
    |\beta_N(x)|\le \frac{c_4 x^{l+2}}{N^r},\qquad N\ge [l+p+3],
\end{equation}
where $r = \min\{l+p+1, 2l+3\}$.
\end{proposition}

\begin{proof}
In \cite{FitouhiHamza} the following asymptotic expansion for $u_l(\omega,x)$ was obtained
\begin{equation}\label{ul asympt}
    u_l(\omega, x) = \sum_{k=0}^m A_k(x) \frac{\sqrt{x}J_{l+k+1/2}(\omega x)}{\omega^{l+k+1/2}}+ \mathcal{R}_{m}(\omega,x).
\end{equation}
The coefficients $A_k$ are defined as follows. Consider the operators $H_p$, $p=1,2,\ldots$, acting as
\[
(H_pf)(x) = \begin{cases}
\frac 1{x^p}\int_0^x t^{p-1} f(t)\,dt, & \text{if } x\ne 0,\\
\frac 1p f(0), &\text{if } x=0.
\end{cases}
\]
One can verify that if $f\in C^{(r)}[0,b]$ then $H_pf\in C^{(r)}[0,b]$ and $(H_pf)^{(r)} = H_{p+r}(f^{(r)})$, $r=0,1,2,\ldots$. Let $A_k(x) =: x^k B_k(x)$. Then the functions $B_k$ satisfy the following recursive relations
\begin{equation}
    B_0 = 2^{l+1/2}\Gamma(l+3/2)\label{B0}\\
\end{equation}
(with such choice of $B_0$ the first term of \eqref{ul asympt} coincides with $d(\omega)b_l(\omega x)$) and
\begin{equation}
    B_{k+1} = -\frac 12 H_{k+1}\bigl[B_k''-2l H_1 B_k''-qB_k\bigr].\label{Bk}
\end{equation}
Moreover, their derivatives satisfy the equalities
\begin{equation}\label{Bk der}
    B_{k+1}^{(j)} = -\frac 12 H_{k+j+1}\bigl[B_k^{(j+2)} - 2l H_{j+1}B_k^{(j+2)} - (qB_k)^{(j)}\bigr],\qquad j\in \mathbb{N}.
\end{equation}
For $\omega\in\mathbb{R}$, $|\omega|\ge 1$ the remainder $\mathcal{R}_{m}$ satisfies the inequality
\begin{equation*}
    |\mathcal{R}_{m}(\omega,x)|\le \frac{c(l,m)}{|\omega|^{l+m+2}}\int_0^x \left| \left( t^{m+1}B_{m+1}(t)\right)'\right|\,dt,
\end{equation*}
and, as can be seen from \eqref{ul asympt}, $\mathcal{R}_{m}(\omega,x)$ is an even entire function of the complex variable $\omega$.

Observe that the condition $q\in C^{2k+j}[0,b]$ is sufficient for the $j$-th derivative of the function $B_{k+1}$ to be well defined and bounded on $[0,b]$ via the formulas \eqref{B0}--\eqref{Bk der}.  Hence it is sufficient to have $q\in C^{2p-2}[0,b]$ to obtain the coefficients $A_0,\ldots,A_{p}$. As for the expression $\bigl(x^{p+1}B_{p+1}(x)\bigr)'$, we show that the condition $q\in C^{2p-1}[0,b]$ is sufficient, one does not need to ask for $q$ to be $2p+1$ times differentiable. Indeed, it follows from \eqref{Bk} that
\[
\bigl(x^{p+1}B_{p+1}\bigr)' = -\frac 12 \Bigl( \bigl(x^pB_p\bigr)'' - 2(p+l)x^{p-1}B_p' - p(p-1) x^{p-2} B_p + 2l x^{p-1}B_p'(0) - x^p qB_p\Bigr),
\]
and the last four terms are well defined whenever $q\in C^{2p-1}[0,b]$, as was explained above. For the first term we consider two cases. If $p=1$ then $(xB_1)''=-\frac{B_0}2 q'(x)$, and for $p>1$ using \eqref{Bk} we have that the expression
\[
\bigl(x^p B_p\bigr)'' = -\frac 12\Bigl( x^{p-1}B_{p-1}'' -2l x^{p-2}B_{p-1}' +2l x^{p-2}B_{p-1}'(0) -x^{p-1}q B_{p-1}\Bigr)'
\]
is well defined whenever $q\in C^{2p-1}[0,b]$.

Applying the Paley-Wiener theorem as in the proof of Proposition \ref{Prop Smooth R} we obtain that
\begin{equation}\label{Rp Int}
\mathcal{R}_p(\omega,x) = \int_{-x}^x \mathcal{\hat R}(x,t) e^{i\omega t}\,dt,
\end{equation}
where the function $\mathcal{\hat R}(x,\cdot)\in W^{l+p+3/2-\varepsilon}(\mathbb{R})\cap \operatorname{Lip}_{l+p+1-\varepsilon}(\mathbb{R})\cap \operatorname{Lip}^*_{l+p+3/2}(\mathbb{R},2)$ and $\operatorname{supp}\mathcal{\hat R}(x,\cdot)\subset[-x,x]$.
From the formula \cite[(5.10.2)]{Lebedev}, $J_\nu(z) = \frac{(z/2)^\nu}{\sqrt{\pi}\Gamma(\nu+1/2)}\int_{-1}^1(1-t^2)^{\nu-1/2}\cos(zt)\,dt$ we have
\begin{equation}\label{Jnu Int}
    \frac{\sqrt x J_{l+k+1/2}(\omega x)}{\omega^{l+k+1/2}} = \frac{x^{l+k}}{\sqrt\pi 2^{l+k+1/2}\Gamma(l+k+1)}\int_{-x}^x \left(1-\frac {z^2}{x^2}\right)^{l+k}e^{i\omega z}\,dz.
\end{equation}
Comparing the expressions \eqref{g omega R}, \eqref{ul asympt}, \eqref{Rp Int} and \eqref{Jnu Int} one can see that
\begin{equation}\label{diff Rs}
\widetilde R(x,t) - \mathcal{\hat R}(x,t) = \sum_{k=1}^p \frac{A_k(x) x^{l+k}}{\sqrt\pi 2^{l+k+1/2}\Gamma(l+k+1)}\cdot \left(1-\frac{t^2}{x^2}\right)^{l+k},\qquad -x\le t\le x.
\end{equation}

Suppose that $l\in\mathbb{N}_0$. Then for every fixed $x$,  the right-hand side of \eqref{diff Rs} is a polynomial in $t$, i.e., a $C^\infty$-function implying that the smoothness of $\widetilde R(x,\cdot)$ coincides with the smoothness of $\mathcal{\hat R}(x,\cdot)$ for $t\in[-x,x]$.
Now one can obtain all remaining statements of the proposition similarly to the proof of Theorem \ref{Th Legendre for R}.

Suppose that $l\not\in\mathbb{N}_0$. Consider the Fourier-Legendre series for $\mathcal{\hat R}$,
\[
\mathcal{\hat R}(x,t)=\frac 12\sum_{n=0}^{\infty}\frac{\hat\beta_{n}(x)}{x}P_{2n}\left(  \frac{t}%
{x}\right).
\]
Similarly to the proof of Theorem \ref{Th Legendre for R} we have that $|\hat\beta_N(x)|\le c x^{l+p+2}(N-1)^{-l-p-1}$ when $2N\ge [l+p+9/2]$. The Fourier-Legendre coefficients for the right-hand side of \eqref{diff Rs} can be estimated using Lemma \ref{Lemma beta k decay special}. Combining the estimates one obtains \eqref{betan estimate NonInt l}. The difference $R-R_N$ can be estimated using \eqref{betan estimate NonInt l} and the orthogonality of the Legendre polynomials. One has
\[
\|R(x,\cdot) - R_N(x,\cdot)\|^2_{L_2[0,x]} = \sum_{n=N+1}^\infty \frac{|\beta_n(x)|^2}{x^2}\cdot \frac{x}{4n+1}\le \frac{c^2_4 x^{2l+3}}{2rN^{2r}}.\qedhere
\]
\end{proof}

\begin{remark}\label{Rmk Saturation}
The estimates  \eqref{L2 estimate R RN NonInt l} and \eqref{betan estimate NonInt l} present a saturation phenomenon, the exponent $r$ can not exceed $2l+3$ whatever smooth the potential $q$ is. It is not a technical restriction, but an intrinsic property of the proposed representation. Indeed, one can deduce similarly to the proof of Lemma \ref{Lemma beta k decay special} that the order $2n+2k+2$ in \eqref{beta k decay special} can not be improved. Hence, whenever one has  $A_1(x)=\frac{A_0}2 Q(x)\ne 0$ in \eqref{diff Rs}, the decay rate order of the coefficients $\beta_n(x)$ for sufficiently smooth potential $q$ is determined by the first term in \eqref{diff Rs} and can not exceed $2l+3$.
\end{remark}

\begin{remark}
Using additional terms from the asymptotic expansion \eqref{ul asympt} together with \eqref{Rp Int} in comparison with representation \eqref{ul} may result in a modification of the method proposed in this paper allowing one to obtain an improved convergence rate. We leave the detailed analysis for a separate paper.
\end{remark}

\section{Representation of the regular solution}

Here we prove the main result of the present work.

\begin{theorem}\label{Th ul Neumann}
Under the conditions of Theorem \ref{Th Legendre for R}, the regular solution $u_{l}(\omega,x)$ of \eqref{Bessel type} satisfying the
asymptotic relation $u_{l}(\omega,x)\sim x^{l+1}$ when $x\rightarrow0$ has the
form%
\begin{equation}
u_{l}(\omega,x)=d(\omega)b_{l}(\omega x)+\sum_{n=0}^{\infty}\left(  -1\right)
^{n}\beta_{n}(x)j_{2n}(\omega x) \label{ul via betas}%
\end{equation}
where $\beta_{n}$ are defined by \eqref{beta n} and $j_{2n}$ stands for the
spherical Bessel function of the first kind of order $2n$, the
series converges uniformly with respect to $x$ on $[0,b]$ and converges
uniformly with respect to $\omega$ on any finite subset of the complex plane
of the variable $\omega$.

For the approximate solution
\begin{equation}
u_{l;N}(\omega,x)=d(\omega)b_{l}(\omega x)+\sum_{n=0}^{N}\left(  -1\right)
^{n}\beta_{n}(x)j_{2n}(\omega x) \label{ulN}%
\end{equation}
the following estimate holds
\begin{equation}
\label{est uln1}\left\vert u_l(\omega,x)-u_{l;N}(\omega,x)\right\vert \leq
\sqrt x\varepsilon_{N}(x)
\end{equation}
for any $\omega\in\mathbb{R}$, $\omega\ne0$, and
\begin{equation}
\label{est uln2}\left\vert u_l(\omega,x)-u_{l;N}(\omega,x)\right\vert \leq
\left(\frac{\sinh(2Cx)}{2C}\right)^{1/2}\varepsilon_{N}(x)
\end{equation}
for any $\omega\in\mathbb{C}$, $\omega\neq0$ belonging to the strip
$\left\vert \operatorname{Im}\omega\right\vert \leq C$, $C\geq0$, where
$\varepsilon_N$ is a sufficiently small nonnegative function such that
$\| R(x,\cdot)-R_{N}(x,\cdot)\|_{L_2[0,x]} \leq\varepsilon_N(x)$ which exists due
to Theorem \ref{Th Legendre for R} (an estimate for
$\varepsilon_N(x)$ is presented in \eqref{L2 estimate R RN}).

Moreover, for each fixed $x$ and $\omega$ the convergence rate of $u_{l;N}(\omega,x)$ to $u_l(\omega,x)$ is exponential. To be more precise, let $x>0$ be fixed and $\omega\in\mathbb{C}$ satisfy $|\omega|\le \omega_0$. Then for all $N>\omega_0 x/2$ one has
\begin{equation}\label{est uln exp}
    |u_l(\omega,x)-u_{l;N}(\omega,x)|\le \frac{cx e^{|\operatorname{Im}\omega|x}}{N^{l+1-\alpha}}\cdot\frac{1}{(2N+2)!}\cdot\left|\frac{\omega_0 x}{2}\right|^{2N+2},
\end{equation}
where $c$ is a constant depending on $q$ and $l$ only and $\alpha$ is the constant from the condition \eqref{Condition precise on q}.
\end{theorem}

\begin{proof}
Consider the solution (\ref{ul}). For the integral from (\ref{ul}) using
formula 2.17.7 from \cite[p. 433]{Prudnikov} we obtain
\begin{equation*}
\int_{0}^{x}R(x,t)\cos\omega t\,dt    =\sum_{n=0}^{\infty}\frac{\beta_{n}%
(x)}{x}\int_{0}^{x}P_{2n}\left(  \frac{t}{x}\right)  \cos\omega t\,dt
  =\sum_{n=0}^{\infty}\left(  -1\right)  ^{n}\beta_{n}(x)j_{2n}(\omega x).
\end{equation*}
Using the Cauchy-Schwarz inequality we obtain  that
\begin{equation*}
\begin{split}
|u_l(\omega;x) - u_{l;N}(\omega;x)| &= \left|\int_0^x \bigl(R(x,t)-R_N(x,t)\bigr)\cos \omega t\,dt\right|\\
&\le \|R(x,\cdot) - R_N(x,\cdot)\|_{L_2[0,x]} \cdot \left(\int_0^x |\cos^2 \omega t|\,dt\right)^{1/2}.
\end{split}
\end{equation*}
The estimates \eqref{est uln1} and \eqref{est uln2} follow immediately from this inequality by taking into account that for complex $\omega$ one has
\[
\int_0^x |\cos^2\omega t|\,dt = \frac 12\int_0^x \left(\cosh (2\operatorname{Im} \omega t)+1\right)\,dt =\frac {\sinh (2|\operatorname{Im} \omega |x)}{4|\operatorname{Im} \omega |}+\frac x2\le \frac {\sinh (2|\operatorname{Im} \omega |x)}{2|\operatorname{Im} \omega |}
\]
and that the function $\sinh t/t$ is monotone increasing for $t>0$.

The uniform convergence with respect to $x$ follows from the estimate \eqref{L2 estimate R RN} for $\varepsilon_N(x)$.

The estimate \eqref{est uln exp} can be obtained using the estimate \eqref{betan est} similarly to the proof of \cite[Proposition 4.4]{KNT 2015}.
\end{proof}

\begin{remark}\label{Rmk Accuracy decay}
Note that the function $u_l(\omega,x)$, as a function of $\omega\in\mathbb{R}$, is bounded by $C\bigl(\frac{x}{b+|\omega|x}\bigr)^{l+1}$, see \cite{KosSakhTesh2010}, i.e., is decaying fast when $\omega\to\infty$, especially for large values of the parameter $l$. Hence, though the uniform estimate \eqref{est uln1} holds for all $\omega\in\mathbb{R}$, for a large $l$ it is practical only for $\omega$ close to zero.
For small values of the parameter $l$ one obtains a quite large region (sufficient to compute hundreds of eigenvalues, for example), while for large values of $l$ the usable region of $\omega$'s is much smaller (suitable to compute a dozen of eigenvalues at most).
\end{remark}

\section{Recurrent equations for $\beta_{n}$\label{Sect Recurrent beta}}

Let us substitute the solution (\ref{ul via betas}) into equation
(\ref{Bessel type}). Definition \eqref{beta n} and estimate \eqref{betan est} together with the inequality \cite[(9.1.62)]{Abramowitz},
$|j_n(x)|\le \sqrt{\pi}\left|\frac x2\right|^n \frac{1}{\Gamma(n+3/2)}$, $x\in\mathbb{R}$ justify the possibility to differentiate the series \eqref{ul via betas} termwise.
Consider
\[
L\left[  \beta_{n}(x)j_{2n}(\omega x)\right]  =\beta_{n}(x)L\left[
j_{2n}(\omega x)\right]  -\beta_{n}^{\prime\prime}(x)j_{2n}(\omega
x)-2\beta_{n}^{\prime}(x)j_{2n}^{\prime}(\omega x)
\]
where the prime means the derivative with respect to $x$. The following
formulas are used
\[
\frac{dj_{k}(z)}{dz}=-j_{k+1}(z)+\frac{k}{z}j_{k}(z)\quad\text{and}\quad
\frac{dj_{k}(z)}{dz}=j_{k-1}(z)-\frac{k+1}{z}j_{k}(z).
\]
Thus,%
\[
j_{2n}^{\prime}(\omega x)=-\omega\left(  j_{2n+1}(\omega x)-\frac{2n}{\omega
x}j_{2n}(\omega x)\right)
\]
and%
\[
j_{2n}^{\prime\prime}(\omega x)=-j_{2n}(\omega x)\left(  \omega^{2}%
-\frac{2n\left(  2n-1\right)  }{x^{2}}\right)  +\frac{2\omega}{x}%
j_{2n+1}(\omega x).
\]
Hence
\[
L\left[  j_{2n}(\omega x)\right]  =j_{2n}(\omega x)\left(  \omega^{2}%
-\frac{2n\left(  2n-1\right)  }{x^{2}}\right)  -\frac{2\omega}{x}%
j_{2n+1}(\omega x)+q_{l}(x)j_{2n}(\omega x)
\]
where $q_{l}(x):=\left(  \frac{l(l+1)}{x^{2}}+q(x)\right)  $ and thus,
\begin{multline*}
0=L\left[  u_{l}(\omega,x)\right]  -\omega^{2}u_{l}(\omega,x)=d(\omega
)q(x)b_{l}(\omega x)\\
+\sum_{n=0}^{\infty}\left(  -1\right)  ^{n}\left[  \beta_{n}(x)\left(
j_{2n}(\omega x)\left(  q_{l}(x)-\frac{2n\left(  2n-1\right)  }{x^{2}}\right)
-\frac{2\omega}{x}j_{2n+1}(\omega x)\right)  \right. \\
\left.  -\beta_{n}^{\prime\prime}(x)j_{2n}(\omega x)+2\beta_{n}^{\prime
}(x)\left(  \omega j_{2n+1}(\omega x)-\frac{2n}{x}j_{2n}(\omega x)\right)
\right]  .
\end{multline*}
We obtain the equality,
which after applying the formula
\begin{equation}
j_{2n}(\omega x)=\frac{\omega x}{4n+1}\bigl(j_{2n-1}(\omega x)+j_{2n+1}(\omega
x)\bigr) \label{recusrive Bessel}%
\end{equation}
can be written in the form
\begin{equation}
\frac{d(\omega)q(x)b_{l}(\omega x)}{\omega x}-q_{l}(x)j_{-1}(\omega
x)=\sum_{n=1}^{\infty}\alpha_{n}(x)j_{2n-1}(\omega x) \label{equality 1}
\end{equation}
with
\begin{equation}\label{alpha n def}
\begin{split}
\alpha_{n}(x)    :=&\left(  -1\right)  ^{n}\left(  \frac{1}{4n+1}\left(
\beta_{n}^{\prime\prime}(x)+\frac{4n}{x}\beta_{n}^{\prime}(x)+\left(
\frac{2n\left(  2n-1\right)  }{x^{2}}-q_{l}(x)\right)  \beta_{n}(x)\right)
\right. \\
&  -\frac{1}{4n-3}\left(  \beta_{n-1}^{\prime\prime}(x)+\frac{4\left(
n-1\right)  }{x}\beta_{n-1}^{\prime}(x)+\left(  \frac{2\left(  n-1\right)
\left(  2\left(  n-1\right)  -1\right)  }{x^{2}}-q_{l}(x)\right)  \beta
_{n-1}(x)\right) \\
&  \left.  +2\left(  \frac{\beta_{n-1}^{\prime}(x)}{x}-\frac{\beta_{n-1}%
(x)}{x^{2}}\right)  \right)  .
\end{split}
\end{equation}
Multiplying equality (\ref{equality 1}) by $j_{2m-1}(\omega x)$,
$m=1,2,\ldots$, integrating with respect to $\omega$ from $0$ to $\infty$ and
using the integrals
\begin{equation*}
\int_{0}^{\infty}j_{2n-1}(\omega x)j_{2m-1}(\omega x)\,d\omega=
\begin{cases}
0, & m\neq n\\
\frac{\pi}{2x\left(  4m-1\right)  },& m=n
\end{cases}
\end{equation*}
for $n,m\in\mathbb{Z}$ with $m+n-1>-1/2$ (c.f., \cite[Formula 11.4.6]{Abramowitz}) we obtain
\[
\frac{\pi\alpha_{n}(x)}{2\left(  4n-1\right)  }=q(x)\int_{0}^{\infty}%
\frac{d(\omega)b_{l}(\omega x)}{\omega}j_{2n-1}(\omega x)d\omega.
\]
Consider the integral
\begin{align*}
\int_{0}^{\infty}\frac{b_{l}(\omega x)}{\omega^{l+2}}j_{2n-1}(\omega x)d\omega
&  =\sqrt{\frac{\pi}{2}}\int_{0}^{\infty}\frac{J_{l+1/2}(\omega x)J_{2n-1/2}%
(\omega x)}{\omega^{l+2}}d\omega\\
&  =\sqrt{\frac{\pi}{2}}\frac{\left(  \frac{x}{2}\right)  ^{l+1}%
\Gamma(l+2)\Gamma(n-1/2)}{2\Gamma(l-n+2)\Gamma(n+1)\Gamma(n+l+3/2)}%
\end{align*}
where formula (1) from \cite[Sect. 13.41]{Watson} was used. Hence%
\[
\int_{0}^{\infty}\frac{d(\omega)b_{l}(\omega x)}{\omega}j_{2n-1}(\omega
x)d\omega=\frac{x^{l+1}\sqrt{\pi}\Gamma(l+2)\Gamma(l+3/2)\Gamma(n-1/2)}%
{4\Gamma(l-n+2)\Gamma(n+1)\Gamma(n+l+3/2)}.
\]
Thus,
\[
\alpha_{n}(x)=B_n q(x)x^{l+1}, \qquad \text{with}\ B_n:=\frac{\left(  4n-1\right)  \Gamma(l+2)\Gamma(l+3/2)\Gamma(n-1/2)}{2\sqrt{\pi}\Gamma(l-n+2)\Gamma
(n+1)\Gamma(n+l+3/2)}.
\]
It is easy to see (c.f., \cite{KNT 2015} and \eqref{alpha n def}) that this equation can be written in the following form
\begin{equation}\label{recurrent eq for beta n}
\frac{1}{\left(  4n+1\right)  x^{2n}}L\left[  x^{2n}\beta_{n}\right]
=\frac{x^{2n-1}}{4n-3}L\left[  \frac{\beta_{n-1}}{x^{2n-1}}\right]
-(-1)^nB_n q(x) x^{l+1}.
\end{equation}
Thus, we obtained a sequence of equations satisfied by the coefficients $\beta_n$.


A solution $u$ of an equation
\begin{equation}\label{eq Lu h}
    Lu=h
\end{equation}
can be easily obtained using the P\'{o}lya factorization of $L$, $Lu = -\frac 1{u_0} \partial u_0^2 \partial \frac u{u_0}$, where $\partial$ denotes the derivative with respect to $x$ and $u_0$ is the same solution as in Section \ref{Sect Transmut Bessel}. The function
\begin{equation}\label{u Polya}
    u(x) = -u_0(x) \int_0^x \frac 1{u_0^2(t)} \int_0^t u_0(s) h(s)\, ds\,dt
\end{equation}
is a solution of \eqref{eq Lu h} provided, e.g., that $|h(x)|\le C x^{l-1+\varepsilon}$ in a neighborhood of zero for some positive $C$ and $\varepsilon$, see \cite{CKT2013} for further details. Note also that the expression \eqref{u Polya} gives the unique solution of \eqref{eq Lu h} satisfying $u(x) = o(x^{l+1})$, $x\to 0$.

One can see from \eqref{Xtilde} that the functions $\varphi_n$ satisfy
$|\varphi_n(x)|\le c_{n,1}x^{2n+l+1}$, $|\varphi_n'(x)|\le c_{n,2}x^{2n+l}$ and $|\varphi_n''(x)|\le c_{n,3}x^{2n+l-1}$, $n\ge 0$, for some constants $c_{n,i}$. Hence it follows from \eqref{beta n} (c.f., \eqref{betan est}) that $|\beta_n(x)|\le c_{n,4}x^{l+1}$, $|\beta_n'(x)|\le c_{n,5}x^{l}$ and $|\beta_n''(x)|\le c_{n,6}x^{l-1}$, $x> 0$. These estimates justify that the expression \eqref{u Polya} can be used to construct the functions $\beta_n$, $n\ge 1$ from equations \eqref{recurrent eq for beta n}.
One starts with
\begin{equation}
\beta_{0}(x)=u_{0}(x)-x^{l+1} \label{beta zero}
\end{equation}
and define for $n\ge 1$
\begin{equation}\label{beta n recursive}
    \beta_n(x) = -\frac{(4n+1)u_0(x)}{x^{2n}}\int_0^x \frac 1{u_0^2(t)} \int_0^t u_0(s) \left[ \frac{s^{4n-1}}{4n-3}L\left(\frac{\beta_{n-1}}{s^{2n-1}}\right)-(-1)^n B_n q(s)s^{2n+l+1}\right]\, ds\,dt.
\end{equation}
Note that for an integer $l$ the second term under the integral sign is equal to zero for $n\ge l+2$. To eliminate the first and second derivatives of $\beta_{n-1}$ resulting from the term $L\left(\frac{\beta_{n-1}}{s^{2n-1}}\right)$, one may apply the integration by parts and obtain (similarly to \cite{KNT 2015}) the following recurrent formulas.
\begin{align}
    \eta_n(x) &= \int_0^x \bigl(tu_0'(t)+(2n-1)u_0(t)\bigr)t^{2n-2}\beta_{n-1}(t)\,dt, &
    \kappa_n(x) &= \int_0^x u_0(t) q(t) t^{2n+l+1}\,dt,\label{etan kappan}\\
    \theta_n(x) &= \int_0^x \frac{1}{u_0^2(t)}\bigl( \eta_n(t) - t^{2n-1} \beta_{n-1}(t) u_0(t)\bigr)\,dt, &
    \mu_n(x) &= \int_0^x \frac{\kappa_n(t)}{u_0^2(t)}\,dt,\qquad n\ge 1,\label{thetan mun}
\end{align}
and finally
\begin{equation}\label{beta_n alt}
    \beta_n(x) = \frac{4n+1}{4n-3}\left[\beta_{n-1}(x) + \frac{u_0(x)}{x^{2n}}\bigl[2(4n-1)\theta_n(x) + (-1)^n (4n-3)B_n \mu_n(x)\bigr]\right].
\end{equation}

\section{Representation of the derivative of the regular solution}

In order to obtain a series
expansion for $u_{l}^{\prime}(\omega,x)$ uniformly convergent with respect to $\omega$ (here and below prime means the
derivative with respect to $x$) we return to (\ref{ul}) from which
\begin{equation}
u_{l}^{\prime}(\omega,x)=\omega d(\omega)b_{l}^{\prime}(\omega x)+\int_{0}^{x}%
R_{1}(x,t)\cos\omega t\,dt \label{ul prime}%
\end{equation}
with the subindex \textquotedblleft$1$\textquotedblright\ denoting the partial
derivative with respect to the first argument. Here we took into account that
by construction $R(x,x)\equiv0$ (see (\ref{R=Const R0}) and (\ref{R(0)}), also Proposition \ref{Prop Smooth R}).

We have
\begin{align}
R_{1}(x,s)  &  =A\frac{\partial}{\partial x}\int_{s}^{x}V(x,t)\left(
t-\frac{s^{2}}{t}\right)  ^{l}\,dt\nonumber\\
&  =A\left(  \frac{Q(x)}{2}x^{-l}\left(  x^{2}-s^{2}\right)  ^{l}+\int_{s}%
^{x}V_{1}(x,t)\left(  t-\frac{s^{2}}{t}\right)  ^{l}\,dt\right)  \label{R1}%
\end{align}
where
\[
A:=\frac{2\Gamma\left(  l+3/2\right)  }{\sqrt{\pi}\Gamma\left(  l+1\right)
}.
\]
To obtain (\ref{R1}) we used (\ref{V(x,x)}).

Denote
\begin{equation}
R^{(1)}(x,s):=\frac{AQ(x)}{2}x^{-l}\left(  x^{2}-s^{2}\right)  ^{l}%
\quad\text{and\quad}R^{(2)}(x,s):=A\int_{s}^{x}V_{1}(x,t)\left(  t-\frac
{s^{2}}{t}\right)  ^{l}\,dt. \label{R1 and R2}%
\end{equation}
Notice that for $-1/2\leq l<0$ the function $R^{(1)}$ has a singularity when
$x=s$. In order to calculate $u_{l}^{\prime}(\omega,x)$ from (\ref{ul prime})
we split the integral into two parts%
\[
\int_{0}^{x}R_{1}(x,t)\cos\omega t\,dt=\int_{0}^{x}R^{(1)}(x,t)\cos\omega
t\,dt+\int_{0}^{x}R^{(2)}(x,t)\cos\omega t\,dt.
\]
Consider
\[
\int_{0}^{x}R^{(1)}(x,t)\cos\omega t\,dt=\frac{AQ(x)}{2}x^{-l}\int_{0}%
^{x}\left(  x^{2}-s^{2}\right)  ^{l}\cos\omega t\,dt.
\]
Due to Remark \ref{Rem Yl[cos]} we obtain that
\begin{equation*}
\int_{0}^{x}R^{(1)}(x,t)\cos\omega t\,dt  =\frac{d\left(  \omega\right)  }{2}Q(x)b_{l}\left(  \omega x\right)  .
\end{equation*}
Thus,
\begin{equation}
u_{l}^{\prime}(\omega,x)=d(\omega)\left(  \omega b_{l}^{\prime}(\omega x)+\frac
{Q(x)}{2}b_{l}\left(  \omega x\right)  \right)  +\int_{0}^{x}R^{(2)}%
(x,t)\cos\omega t\,dt. \label{ul prime R2}%
\end{equation}

The article \cite{Volk} does not provide sufficient details on the behavior of the derivative $V_1(x,t)$ near $t=0$. As a result, the definition \eqref{R1 and R2} is not quite helpful for studying the integral kernel $R^{(2)}$, even the integrability of $R^{(2)}(x,t)$ near $t=0$ necessary for the representation \eqref{ul prime R2} to be well defined goes under the question. Below we provide a different proof of the representation \eqref{ul prime R2} based on the Paley-Wiener theorem and similar to that of Proposition \ref{Prop Smooth R}.
\begin{theorem}
\label{Th R2 Legendre}Assume additionally to the conditions of Theorem \ref{Th u(omega,x)} that $q\in C^1[0,b]$. Let $x>0$ be fixed.
Then there exists an even, compactly supported on $[-x,x]$ function $\widetilde R^{(2)}(x,t)$ such that $\widetilde R^{(2)}(x,\cdot)\in W_2^{l+3/2-\varepsilon}(\mathbb{R})\cap \operatorname{Lip}_{l+1-\varepsilon}(\mathbb{R})\cap \operatorname{Lip}^*_{l+3/2}(\mathbb{R})$ for any small $\varepsilon>0$ and the representation \eqref{ul prime R2} holds with the function $R^{(2)}$ satisfying $R^{(2)}(x,t) = 2\widetilde R^{(2)}(x,t)$, $0\le t\le x$.

The kernel $R^{(2)}$ from (\ref{R1 and R2}) admits the
following representation%
\begin{equation}
R^{(2)}(x,t)=\sum_{n=0}^{\infty}\frac{\gamma_{n}(x)}{x}P_{2n}\left(  \frac
{t}{x}\right)  \label{R2 Legendre}%
\end{equation}
where
\begin{equation}
\gamma_{n}(x)=\left(  4n+1\right)  \sum_{k=0}^{n}\frac{l_{2k,2n}}{x^{2k}%
}\left(  \varphi_{k}^{\prime}(x)-c_{k,l}\left(  \left(  2k+l+1\right)
x^{2k+l}+\frac{Q(x)}{2}x^{2k+l+1}\right)  \right)  . \label{gamma n}%
\end{equation}
The convergence of the series in \eqref{R2 Legendre} is the same as stated in Theorem \ref{Th Legendre for R} for the series \eqref{R(x,t)}.

Denote
\[
R_{N}^{(2)}(x,t):=\sum_{n=0}^{N}\frac{\gamma_{n}(x)}%
{x}P_{2n}\left(  \frac{t}{x}\right).
\]
Then there exist constants $C_1$ and $C_2$, dependent on $q$ and $l$ and independent of $x$ and $N$, such that for any $x>0$ the inequalities hold
\begin{equation}\label{L2 estimate R2 R2N}
    \|R^{(2)}(x,\cdot) - R^{(2)}_N(x,\cdot)\|_{L_2[0,x]} \le \frac{C_1 x^{l+3/2}}{N^{l+3/2}},\qquad 2N\ge [l+5/2]
\end{equation}
and
\begin{equation}\label{gamman est}
    |\gamma_N(x)|\le \frac{C_2 x^{l+2}}{(N-1)^{l+1}},\qquad 2N\ge [l+9/2].
\end{equation}

Let $q\in C^{2p+1}[0,b]$ with $p\ge 1$. If $l\in \mathbb{N}_0$, then the smoothness of $\widetilde R^{(2)}$ and the order of $N$ in the estimates \eqref{L2 estimate R2 R2N} and \eqref{gamman est} can be increased by $p$; if $l\not\in\mathbb{N}_0$, then the order of $N$ in the estimates \eqref{L2 estimate R2 R2N} and \eqref{gamman est} can be increased to $r = \min\{l+p+1, 2l+3\}$,  c.f.,  Proposition \ref{Prop Int l}
\end{theorem}

\begin{proof}
Consider the asymptotic expansion \eqref{ul asympt} with $m=1$. According to \cite{FitouhiHamza} the remainder $\mathcal{R}_1(\omega, x)$ satisfies the integral equation
\begin{equation}\label{R1 eq}
\mathcal{R}_1(\omega,x) = \int_0^x G_l(\omega, x,t) \left( q(t) \mathcal{R}_1(\omega, t) + 2A_2'(t)\frac{\sqrt{t} J_{l+3/2}(\omega t)}{\omega^{l+3/2}}\right)\,dt,
\end{equation}
where (c.f., \cite{KosSakhTesh2010})
\begin{equation*}
    G_l(\omega, x, t) = -\frac{\pi}2 \sqrt{xt} \bigl(J_{l+1/2}(\omega x) Y_{l+1/2}(\omega t) - J_{l+1/2}(\omega t) Y_{l+1/2}(\omega x)\bigr)
\end{equation*}
and
\[
A_2'(x) = -\frac 12 \left( A_1''(x) - \frac{2(l+1)}{x}A_1'(x) + \left(\frac {2(l+1)}{x^2}-q(x)\right)A_1(x)\right).
\]
Observe that $A_1(x) = \frac{A_0}2 Q(x)$, hence
\[
\begin{split}
A_2'(x) &= -\frac{A_0}2\left(\frac{q'(x)}2 - \frac{l+1}x q(x) + \frac{l+1}{x^2}Q(x) - \frac{q(x)Q(x)}2\right)\\
& = -\frac{A_0}2 \left(\frac{q'(x)}2 - \frac{l+1}{x^2}\int_0^x tq'(t)\,dt - \frac {q(x)Q(x)}2\right).
\end{split}
\]
Since $q\in C^1[0,b]$, $A_2'\in C(0,b]$ and can be extended to a $C[0,b]$ function by continuity. Denote $A_2:=\max_{[0,b]}|A_2'(x)|$.

The derivative $\mathcal{R}_{1,x}:=\frac{\partial}{\partial x}\mathcal{R}_1$ satisfies
\begin{equation}\label{R1x eq}
    \mathcal{R}_{1,x} (\omega, x) = \int_0^x \frac{\partial}{\partial x} G_l(\omega, x, t) \left( q(t) \mathcal{R}_1(\omega, t) + 2A_2'(t)\frac{\sqrt{t} J_{l+3/2}(\omega t)}{\omega^{l+3/2}}\right)\,dt.
\end{equation}
Recall the following estimates for the functions $G_l$ and $J_{l+3/2}$ (see \cite[Appendix 1]{KosSakhTesh2010}), here $\omega\in\mathbb{R}$ and $t\le x$.
\begin{align}
    \left|G_l(\omega, x,t)\right| & \le C\left(\frac{x}{b+|\omega|x}\right)^{l+1}\left(\frac{b+|\omega|t}{t}\right)^{l}\theta(t),\label{Gl est}\\
    \left|\frac{\partial}{\partial x}G_l(\omega, x,t)\right| &\le C\left(\frac{x}{b+|\omega|x}\right)^{l}\left(\frac{b+|\omega|t}{t}\right)^{l}\theta(t),\label{Glx est}
\end{align}
where
$\theta(t) = 1$ for $l>-1/2$ and $\theta(t) = 1-\log (t/b)$ for $l=-1/2$, and
\begin{equation}\label{J est}
    \sqrt{x}\left|J_{l+3/2}(\omega x)\right| \le C_1\sqrt{\frac 2\pi} \omega^{l+3/2}\left(\frac{x}{b+|\omega|x}\right)^{l+2}.
\end{equation}
Applying the successive approximations method similarly to  \cite[Chap. 6, \S10]{Olver}) one easily obtains that
\begin{equation}\label{R1 and R1x estimates}
    |\mathcal{R}_1(\omega, x)|\le \frac{\widetilde C}{|\omega|^2}\left(\frac x{b+|\omega| x}\right)^{l+1}\qquad\text{and}\qquad |\mathcal{R}_{1,x}(\omega, x)|\le \frac{\widetilde C_1}{|\omega|^2}\left(\frac x{b+|\omega| x}\right)^{l}.
\end{equation}
Indeed, consider
\[
r_0(x) := \int_0^x G_l(\omega, x,t) 2A_2'(t)\frac{\sqrt{t} J_{l+3/2}(\omega t)}{\omega^{l+3/2}}\,dt \qquad \text{and}\qquad r_{n+1}(x) := \int_0^x G_l(\omega, x,t) q(t) r_n(t)\,dt.
\]
Then using the inequalities \eqref{Gl est}, \eqref{Glx est}, \eqref{J est} we obtain that
\[
|r_0(x)| \le 2A_2CC_1\sqrt{\frac{2}{\pi}}\left(\frac x{b+|\omega| x}\right)^{l+1}\int_0^x \left(\frac t{b+|\omega| t}\right)^{2}\theta(t)\,dt\le \frac{C_2}{\omega^2}\left(\frac x{b+|\omega| x}\right)^{l+1},
\]
where we used that $\frac {t}{b+|\omega|t}\le \frac 1{|\omega|}$ and  $\int_0^x \theta(t)\,dt \le \int_0^b\theta(t)\,dt <\infty$. It follows by induction that
\[
|r_n(x)| \le \frac{C_2}{n!|\omega|^2} \left(\frac x{b+|\omega| x}\right)^{l+1}\left(C\int_0^x \frac {t\tilde q(t)}{b+|\omega| t}\,dt\right)^n,
\]
where $\tilde q(t)$ is the same as in the proof of Proposition \ref{Prop Smooth R}.
Summing up all the functions $r_n$ we obtain the first estimate in \eqref{R1 and R1x estimates} with $\tilde C:= C_2\exp\left(C\int_0^b \frac {t\tilde q(t)}{b+|\omega| t}\,dt\right)<\infty$. The second estimate follows directly from the first estimate and \eqref{R1x eq}.

Differentiating \eqref{ul asympt} with respect to $x$ we obtain that
\begin{equation}\label{R1x}
\mathcal{R}_{1,x}(\omega, x) = u_l'(\omega, x) - d(\omega) \left(\omega b_l'(\omega x) + \frac{Q(x)}2 b_l(\omega x)\right) - \frac{A_0 \sqrt{x}J_{l+3/2}(\omega x)}{4\omega^{l+3/2}}\left(2q(x) - \frac{(2l+1)Q(x)}{x}\right),
\end{equation}
where we used the formula $J_\nu'(z) = J_{\nu-1}(z)-\frac{\nu}z J_\nu(z)$. Consider the function
\[
g_2(\omega):= u_l'(\omega, x) - d(\omega) \left(\omega b_l'(\omega x) + \frac{Q(x)}2 b_l(\omega x)\right).
\]
As follows from \eqref{R1 and R1x estimates}, $|\mathcal{R}_{1,x}(\omega,x)|\le \frac{\widetilde C_1}{\omega^{l+2}}$. The last term in \eqref{R1x} decays as $\omega^{-l-2}$ when $|\omega|\to\infty$. Hence the function $g_2(\omega)$ also decays as $\omega^{-l-2}$. Moreover, $g_2(\omega)$ is an entire even function of the complex variable $\omega$ and similarly to the proof of Proposition \ref{Prop Smooth R} we obtain the existence and smoothness of the function $\widetilde R^{(2)}$.

From (\ref{R2 Legendre}) similarly to the proof of Theorem
\ref{Th Legendre for R} we obtain
\begin{equation}
\gamma_{n}(x)    =\left(  4n+1\right)  \int_{0}^{x}R^{(2)}(x,t)P_{2n}\left(
\frac{t}{x}\right)  \,dt =\left(  4n+1\right)  \sum_{k=0}^{n}\frac{l_{2k,2n}}{x^{2k}}\int_{0}%
^{x}R^{(2)}(x,t)t^{2k}\,dt. \label{gamma via int}%
\end{equation}
In order to calculate the last integral we expand all the terms of the equality \eqref{ul prime R2} into the series with respect to $\omega$ (for the function $u_l'$ we differentiate \eqref{u via SPPS}) and compare coefficients at equal powers of $\omega$. We obtain that
\[
\int_{0}^{x}R^{(2)}(x,t)t^{2k}\,dt=\varphi_{k}^{\prime}(x)-c_{k,l}%
x^{2k+l}\left(  \left(  2k+l+1\right)  +\frac{xQ(x)}{2}\right)  .
\]
Substitution of this expression into (\ref{gamma via int}) gives us
(\ref{gamma n}).

Convergence of the series \eqref{R2 Legendre} and the estimates \eqref{L2 estimate R2 R2N} and \eqref{gamman est} can be obtained similarly to the proof of Theorem \ref{Th Legendre for R}.

Suppose that $q\in C^{2p+1}[0,b]$. Consider the asymptotic expansion \eqref{ul asympt} with $m=p+1$. The remainder $\mathcal{R}_{p+1}$ satisfies equations similar to \eqref{R1 eq} and \eqref{R1x eq}. Applying the successive approximations method one can obtain the following estimate for the derivative $|\mathcal{R}_{p+1,x}(\omega, x)|\le \frac {C_3}{|\omega|^{p+2}}\Bigl(\frac{x}{b+|\omega|x}\Bigr)^{l}$.
Observe that
$\Bigl(A_k(x)\frac{\sqrt x J_{l+k+1/2}(\omega x)}{\omega^{l+k+1/2}}\Bigr)'= \Bigl(A_k'(x)\sqrt x- \frac{(l+k)A_k(x)}{\sqrt x}\Bigr)\frac{J_{l+k+1/2}(\omega x)}{\omega^{l+k+1/2}} + A_k(x)\sqrt x \frac{J_{l+k-1/2}(\omega x)}{\omega^{l+k-1/2}}$, i.e., for each fixed $x$ the expressions $g_2$ and $\mathcal{R}_{p+1,x}$ differ by a linear combination of the terms $\frac{J_{l+k+1/2}(\omega x)}{\omega^{l+k+1/2}}$. Now the last statement of the theorem can be obtained following the proof of Proposition \ref{Prop Int l}.
\end{proof}

\begin{theorem}Under the conditions of Theorem \ref{Th R2 Legendre}, the $x$-derivative of the regular solution $u_{l}(\omega,x)$ of
\eqref{Bessel type} satisfying the asymptotic relation $u_{l}(\omega,x)\sim
x^{l+1}$ when $x\rightarrow0$ has the form
\begin{equation}
u_{l}^{\prime}(\omega,x)=d(\omega)\left(  \omega b_{l}^{\prime}(\omega x)+\frac
{Q(x)}{2}b_{l}\left(  \omega x\right)  \right)  +\sum_{n=0}^{\infty}\left(
-1\right)  ^{n}\gamma_{n}(x)j_{2n}(\omega x) \label{ul prime via gammas}%
\end{equation}
where the coefficients $\gamma_{n}$ are defined by \eqref{gamma n}. For the
difference between $u_{l}^{\prime}(\omega,x)$ and
\begin{equation}\label{ul prime N}
\overset{\circ}{u}
_{l,N}(\omega,x):=d(\omega)\left(  \omega b_{l}^{\prime}(\omega x)+\frac{Q(x)}%
{2}b_{l}\left(  \omega x\right)  \right)  +\sum_{n=0}^{N}\left(  -1\right)
^{n}\gamma_{n}(x)j_{2n}(\omega x)
\end{equation}
the following inequalities are valid
\begin{equation}
\left\vert u_{l}^{\prime}(\omega,x)-\overset{\circ}{u}_{l,N}(\omega
,x)\right\vert \leq \sqrt{x} \varepsilon_N(x)\qquad\text{for all }\omega\in\mathbb{R},
\label{diff real}
\end{equation}
and
\begin{equation}
\left\vert u_{l}^{\prime}(\omega,x)-\overset{\circ}{u}_{l,N}(\omega
,x)\right\vert \leq \left(\frac{\sinh(2Cx)}{2C}\right)^{1/2}\varepsilon_{N}(x)
\qquad\text{for all }\omega\in\mathbb{C},\ \left\vert \operatorname*{Im}\omega\right\vert \leq C,\ C\geq0,
\label{diff complex}
\end{equation}
where
$\varepsilon_N$ is a sufficiently small nonnegative function such that
$\| R^{(2)}(x,\cdot)-R_{N}^{(2)}(x,\cdot)\|_{L_2[0,x]} \leq\varepsilon_N(x)$, which exists due to
Theorem \ref{Th R2 Legendre}.
\end{theorem}

\begin{proof}
Substitution of (\ref{R2 Legendre}) into (\ref{ul prime R2}) together with the
formula 2.17.7 from \cite[p. 433]{Prudnikov} gives us
(\ref{ul prime via gammas}). The inequalities (\ref{diff real}) and
(\ref{diff complex}) are obtained in a complete analogy
with the proof of Theorem \ref{Th ul Neumann}.
\end{proof}

\section{Recurrent equations for $\gamma_{n}$}

From (\ref{ul via betas}) we have that
\[
u_{l}^{\prime}(\omega,x)=\omega d(\omega)b_{l}^{\prime}(\omega x)+\sum_{n=0}^{\infty
}\left(  -1\right)  ^{n}\left(  \beta_{n}^{\prime}(x)j_{2n}(\omega
x)-\omega\beta_{n}(x)j_{2n+1}(\omega x)+\frac{2n}{x}\beta_{n}(x)j_{2n}(\omega
x)\right)  .
\]
Comparing this expression with (\ref{ul prime via gammas}) we obtain the
equality%
\begin{multline*}
\frac{d(\omega)Q(x)}{2}b_{l}\left(  \omega x\right)  +\sum_{n=0}^{\infty
}\left(  -1\right)  ^{n}\gamma_{n}(x)j_{2n}(\omega x)\\
=\sum_{n=0}^{\infty
}\left(  -1\right)  ^{n}\left(  \beta_{n}^{\prime}(x)j_{2n}(\omega
x)-\omega\beta_{n}(x)j_{2n+1}(\omega x)+\frac{2n}{x}\beta_{n}(x)j_{2n}(\omega
x)\right)  .
\end{multline*}
Using (\ref{recusrive Bessel}) and rearranging the terms we arrive at the
equality%
\begin{equation}
\frac{d(\omega)Q(x)b_{l}(\omega x)}{2\omega x}-\left(  \beta_{0}^{\prime
}(x)-\gamma_{0}(x)\right)  j_{-1}(\omega x)=\sum_{n=1}^{\infty}\tilde\alpha
_{n}(x)j_{2n-1}(\omega x) \label{equality 2}%
\end{equation}
where
\[
\tilde\alpha_{n}:=\left(  -1\right)  ^{n}\left[  \frac{1}{4n+1}\left(  \beta
_{n}^{\prime}-\gamma_{n}+\frac{2n}{x}\beta_{n}\right)  -\frac{1}{4n-3}\left(
\beta_{n-1}^{\prime}-\gamma_{n-1}-\frac{2n-1}{x}\beta_{n-1}\right)  \right]
\]
for $n=1,2,\ldots$ and $\beta_{0}^{\prime}-\gamma_{0}=x^{l+1}Q/2$.

Comparison of (\ref{equality 2}) with (\ref{equality 1}) and application of a
similar procedure to that from Section \ref{Sect Recurrent beta} leads to the
relations%
\begin{align*}
\frac{\pi\tilde\alpha_{n}(x)}{\left(  4n-1\right)  }  &  =Q(x)\int_{0}^{\infty}%
\frac{d(\omega)b_{l}(\omega x)}{\omega}j_{2n-1}(\omega x)d\omega\\
&  =\frac{Q(x)}{4}\frac{x^{l+1}\sqrt{\pi}\Gamma(l+2)\Gamma(l+3/2)\Gamma
(n-1/2)}{\Gamma(l-n+2)\Gamma(n+1)\Gamma(n+l+3/2)}%
\end{align*}
and hence
\begin{equation}\label{alpha n tilde eq}
\tilde\alpha_{n}(x)=\frac{\left(  4n-1\right)  }{4\sqrt{\pi}}\frac{Q(x)x^{l+1}%
\Gamma(l+2)\Gamma(l+3/2)\Gamma(n-1/2)}{\Gamma(l-n+2)\Gamma(n+1)\Gamma
(n+l+3/2)} =:C_nQ(x) x^{l+1}.
\end{equation}

Equations \eqref{alpha n tilde eq} together with the recursive formulas \eqref{etan kappan}--\eqref{beta_n alt} can be used to calculate the coefficients $\gamma_n$ alternatively to the formulas \eqref{gamma n}. We start with
\begin{equation}\label{gamma 0}
\gamma_0 = \beta_0' + x^{l+1}Q/2 = u_0' -(l+1)x^l + x^{l+1}Q/2
\end{equation}
and compute recursively for $n\ge 1$
\begin{equation}\label{gamma_n alt}
\begin{split}
\gamma_n &= \frac{4n+1}{4n-3}\left[\gamma_{n-1} + (4n-1)\left( \frac{2u_0'\theta_n}{x^{2n}} + \frac{2\eta_n}{u_0 x^{2n}} - \frac{\beta_{n-1}}{x}\right)\right] \\
&\quad+(-1)^n(4n+1)\left[\frac{B_n}{x^{2n}}\left(\mu_n u_0'+\frac{\kappa_n}{u_0}\right)  - C_n Q(x)x^{l+1}\right].
\end{split}
\end{equation}

\section{Numerical results}
The main ingredients for the construction and application of the approximate solution $u_{l;N}$ and its approximate derivative $\overset{\circ}{u}_{l,N}$ are the coefficients $\beta_n$ and $\gamma_n$. Unfortunately, we are not aware of any single non-zero potential $q$ for which one can obtain these coefficients in a closed form. They have to be calculated numerically. Our experiments show that even hundreds of the coefficients $\beta_n$ and $\gamma_n$ can be easily computed within seconds without any difficulty. Below we explain some details. We also refer the reader to \cite{CKT2013}, \cite{KT AnalyticApprox} and \cite{KNT 2015} where many aspects of the numerical implementation are discussed in detail.

The first coefficients $\beta_0$ and $\gamma_0$ are given by \eqref{beta zero} and \eqref{gamma 0} in terms of the particular solution $u_0$ of equation \eqref{PartSolEq} satisfying asymptotic conditions \eqref{SolAsymptotic} and \eqref{DerSolAsymptotic}. Such solution together with its derivative can be computed using the SPPS representation \cite[Section 3]{CKT2013}. The assumption for the solution $u_0$ to be non-vanishing automatically holds if $q(x)\ge 0$, $x\in (0,b]$. For other cases one may need to apply the spectral shift technique as described in \cite{CKT2013}, \cite{CKT2015}.

As was mentioned in \cite{KNT 2015} for the non-singular case, the direct formulas lead to a rapid growth of the error in the computed coefficients $\beta_k$ and $\gamma_k$ and are not recommended for numerics. The same happens with the formulas \eqref{beta n} and \eqref{gamma n}, they allow one to calculate only 10--15 coefficients $\beta_k$ and $\gamma_k$ in the machine precision. However one still may apply them when arbitrary precision arithmetics is used. In the present paper we neither utilize nor present any illustration of the numerical performance of the formulas \eqref{beta n} and \eqref{gamma n}.

Instead, the recurrent formulas \eqref{etan kappan}--\eqref{beta_n alt} and \eqref{gamma_n alt} show an excellent computational stability allowing one to compute easily even hundreds of the coefficients. All the functions involved were represented by their values on the uniform mesh. We used a somewhat overwhelming number of mesh points (like 20--50 thousands) in order to make the integration errors negligible and to concentrate mainly on the numerical performance of the proposed formulas. It is worth emphasizing that even in this case all the reported calculations took only several seconds. The integrals in \eqref{etan kappan} and \eqref{thetan mun} were calculated using the modified 6 point Newton-Cottes rule. This rule consists in interpolating the function values at these 6 points by a fifth order polynomial and using the integral of this polynomial as the approximation for the indefinite integral.

We would like to point out that the numerical integration in \eqref{thetan mun} may be tricky due to the division by $u_0^2$, a function behaving near zero as $x^{2l+2}$. Even small errors in the values of the functions $\eta_n$ and $\kappa_n$ near zero can lead to large erroneous values after dividing by $u_0^2$. As a workaround we chose the following strategy. We simply ignored (replaced by 0) several first values of the integrands in \eqref{thetan mun} prior to numerical integration. Due to 6 point integration rule utilized, we used the following simple cut-off criterion. For every 6 integrand values $y_0,\ldots,y_5$ on the consecutive mesh points we calculated the expression
\[
\Delta_5:= y_0-5y_1+10y_2-10y_3+5y_4-y_5
\]
(related to the fifth order divided difference) and compared it to the two smallest absolute values of the numbers $y_0,\ldots,y_5$. We started the integration rule from the first 6-tuple for which the quantity $|\Delta_5|$ was not significantly larger than the two smallest absolute values. This simple criterion resulted to be sufficient to deliver acceptable numerical results.

\subsection{Analysis of the decay rate of the coefficients $\beta_n$ and $\gamma_n$}\label{Subsect decay analysis}
Absolute values of the coefficients $\beta_k$ decrease as $k\to\infty$, see the estimates \eqref{betan est}, \eqref{betan estimate Int l} and \eqref{betan estimate NonInt l}.
However due to the presence of the term $\frac{4n+1}{4n-3}\beta_{n-1}$ in the formula \eqref{beta_n alt} the error in one computed coefficient $\beta_n$ propagates to all further coefficients. I.e., when one computes large numbers of the coefficients $\beta_n$, their absolute values reach some floor value and stabilize. The same happens with the coefficients $\gamma_n$. A simple error measure can be derived taking $t=x$ in \eqref{R(x,t)} and \eqref{R2 Legendre}. One has
\begin{equation}\label{Check for precision}
    \sum_{n=0}^\infty \frac{\beta_n(x)}{x} = R(x,x) = 0\qquad \text{and}\qquad
    \sum_{n=0}^\infty \frac{\gamma_n(x)}{x} = R^{(2)}(x,x) = 0,
\end{equation}
and the discrepancy of the truncated series from zero provides some insight on how good the approximation is.

Consider $b=\pi$ and $q=x^2$ in \eqref{I1}. We computed the coefficients $\beta_n$ for $n\le 100$ for several different values of the parameter $l$. On Figure \ref{DecayFig1} we present the plot of the values $|\beta_n(\pi)|$ vs.\ $n$. We chose log-log scale graph to reveal a possible power law decay rate of the coefficients. As one can see from the graph, the absolute values $|\beta_n(\pi)|$ indeed obey a power law decay whenever $l\not\in\mathbb{N}$, and a faster than polynomial decay for $l\in\mathbb{N}$, c.f., Proposition \ref{Prop Int l}. We  estimated the decay rate degree $r$ in the power law $|\beta_n(\pi)|\approx c n^{-r}$ and obtained that $r\approx 2l+3$ for $l\not\in\mathbb{N}$, c.f., \eqref{betan estimate NonInt l}.

\begin{figure}[htb!]
\centering
\includegraphics[bb=126 287 486 504, width=5in,height=3in]
{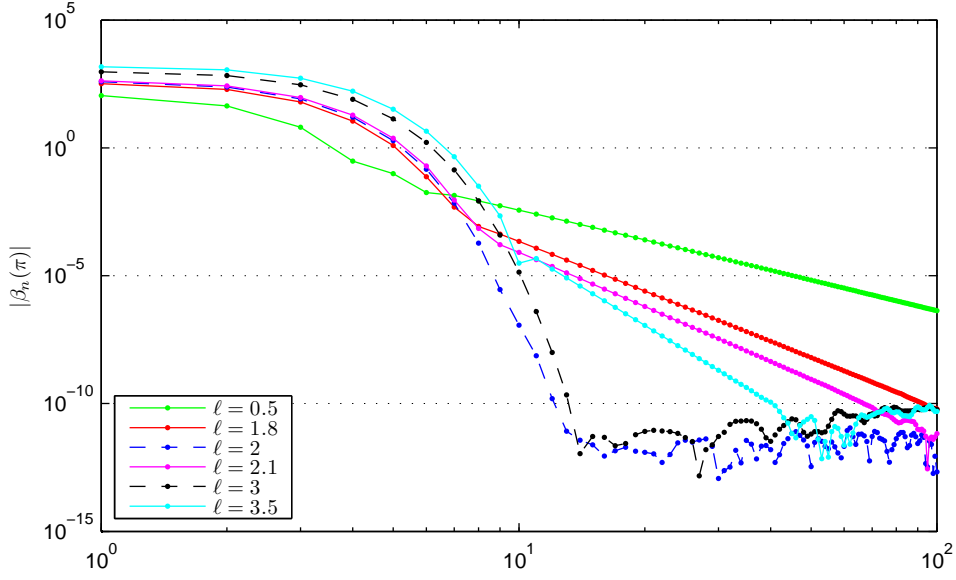}
\caption{Illustration of the decay of the numbers $|\beta_n(\pi)|$ for $q(x)=x^2$ and for different values of the parameter $l$ in \eqref{I1}. Dashed lines correspond to integer values of $l$, solid lines correspond to non-integer values of $l$.}
\label{DecayFig1}
\end{figure}

However we observed that the smoothness requirements on the potential $q$ in Proposition \ref{Prop Int l} and Theorem \ref{Th R2 Legendre} look to be excessive. For that we considered several potentials,
\begin{equation}\label{DecayPotentials}
q_1(x)=x^2,
\qquad
q_2(x) = \sqrt{\pi^2-x^2},
\qquad
q_3(x) = \frac 1x, \qquad
q_k(x) =\begin{cases}
0, & x\le \pi/2,\\
(x-\pi/2)^{k-4}, & x>\pi/2,\\
& k\in\{4,5,6\},
\end{cases}
\end{equation}
computed the coefficients $\beta_n$ and $\gamma_n$, $n\le 100$ for different non integer values of $l$ and found the degrees $r$ and $s$ in the power law approximations $|\beta_n(\pi)|\approx c_1 n^r$ and $|\gamma_n(\pi)|\approx c_2 n^s$. The obtained values of $r$ and $s$ are presented on Figure \ref{DecayFig2}.

The first potential illustrates that one can not expect an improvement of the estimate \eqref{betan estimate NonInt l} even for infinitely smooth potentials (c.f., Remark \ref{Rmk Saturation}). However, it is illustrated by the potentials $q_2$ and $q_3$ that the coefficients $\beta_n$ and $\gamma_n$ can decay as (or closely to) $n^{-2l-3}$ even for potentials possessing singularities or unbounded derivatives at the endpoints. The situation changes when the potentials are not sufficiently smooth inside the interval $(0,b)$, as illustrated by $q_4$, $q_5$ and $q_6$. For small values of the parameter $l$ the coefficients $\beta_n$ and $\gamma_n$ still decay as $n^{-2l-3}$, while for larger values of $l$ the decay rate degree becomes smaller.

\begin{figure}[htb!]
\centering
Decay rate degree of the numbers $|\beta_n(\pi)|$\\
\includegraphics[bb=198 309 414 482, width=3in,height=2.4in]
{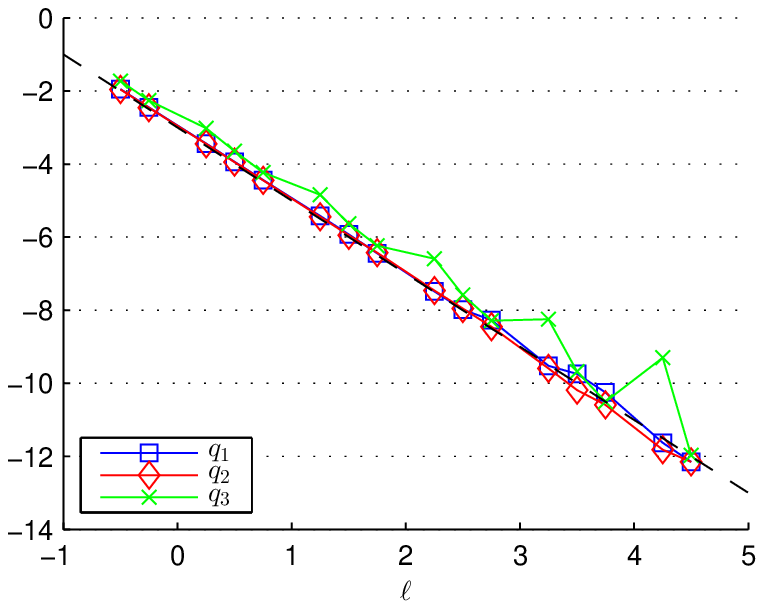}\quad
\includegraphics[bb=198 309 414 482, width=3in,height=2.4in]
{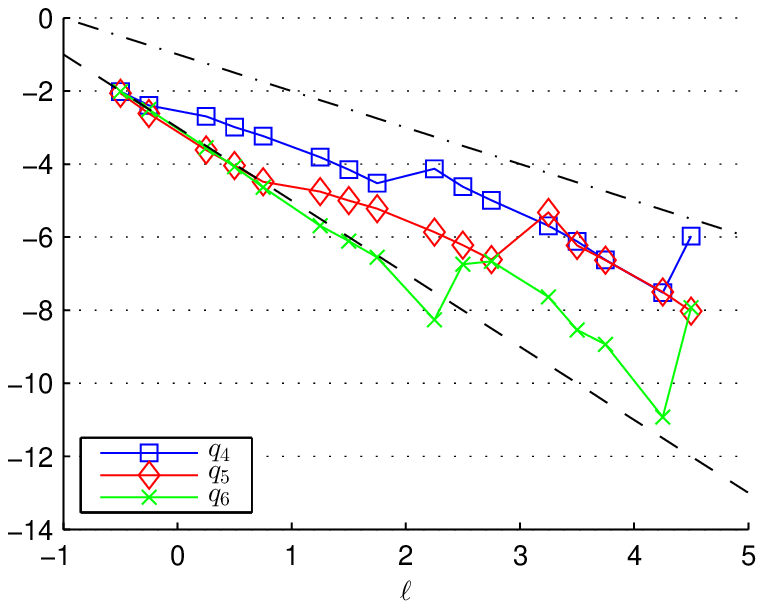}

Decay rate degrees of the numbers $|\gamma_n(\pi)|$\\
\includegraphics[bb=198 309 414 482, width=3in,height=2.4in]
{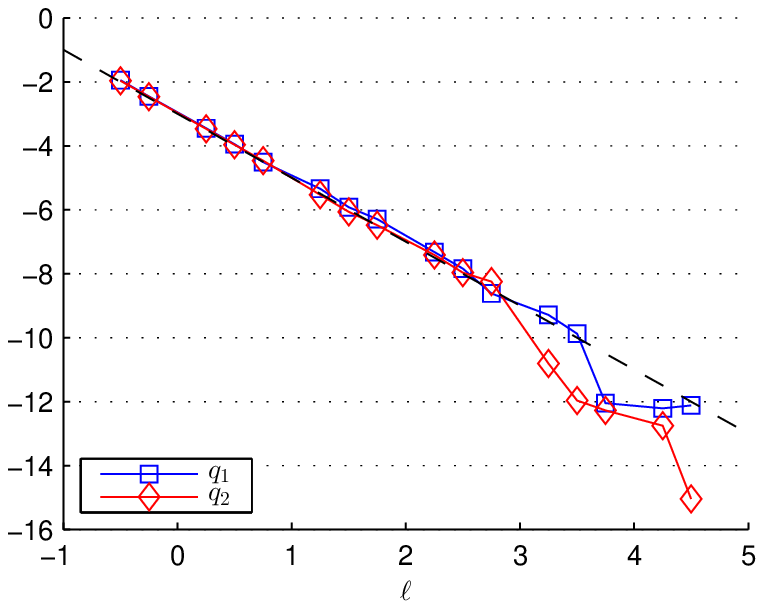}\quad
\includegraphics[bb=198 309 414 482, width=3in,height=2.4in]
{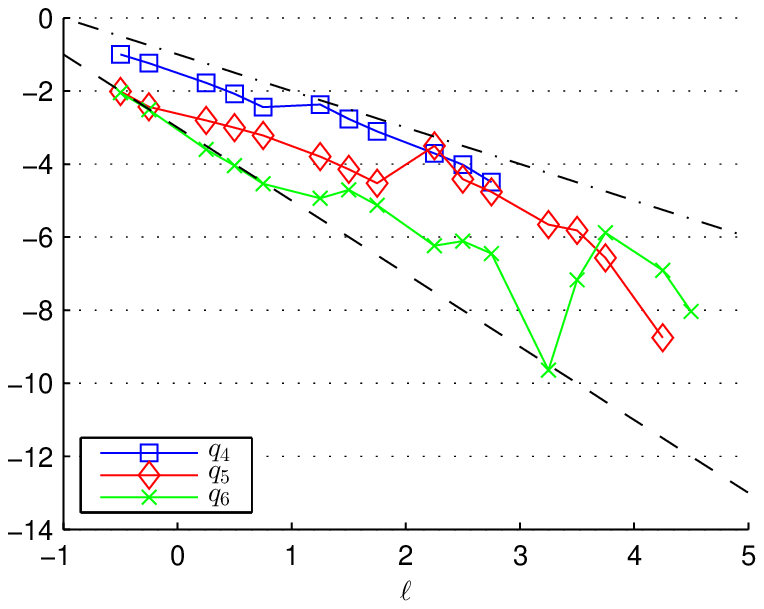}
\caption{Illustration of the decay rate degree of the numbers $|\beta_n(\pi)|$ (two upper plots) and $|\gamma_n(\pi)|$ (two lower plots) for different potentials $q_i$, $i=1,\ldots ,6$ given by \eqref{DecayPotentials} as functions of the parameter $l$. Additionally we plot the lines $y=-2l-3$ (dashed line) and $y=-l-1$ (dash-dot line).}
\label{DecayFig2}
\end{figure}

For integer values of $l$ the estimate \eqref{betan estimate Int l} predicts a faster than polynomial decay of the coefficients $\beta_n$ in the case of a $C^\infty$-potential and guaranties a polynomial decay rate for potentials of finite smoothness. We verified this numerically considering the potentials
\begin{equation}\label{DecayPotentials2}
q_k(x)= \begin{cases}
1, & x\le \pi/2,\\
1+(x-\pi/2)^k, & x> \pi/2,\quad k=0,\ldots,5,
\end{cases}
\end{equation}
and comparing the degree of decay rate as in the previous experiments. On Figure \ref{DecayFig3} we present the plots of the values $|\beta_n(\pi)|$ and $|\gamma_n(\pi)|$ vs.\ $n$. For all calculations we took $l=2$. As one can see from the plots, the slope of the lines (corresponding to the decay rate degree) indeed increases when the potential smoothness increases by 2, as predicted by Proposition \ref{Prop Int l}, however as in the previous experiment, the increase of the slope is by 2, not by 1.

A more detailed study of the observed phenomena is left for a future work.

\begin{figure}[htb!]
\centering
\includegraphics[bb=198 309 414 482, width=3in,height=2.4in]
{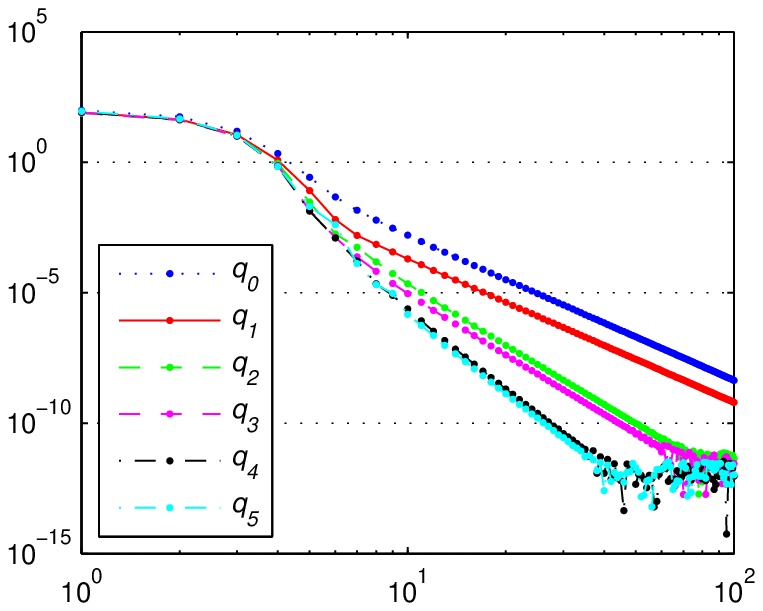}\quad
\includegraphics[bb=198 309 414 482, width=3in,height=2.4in]
{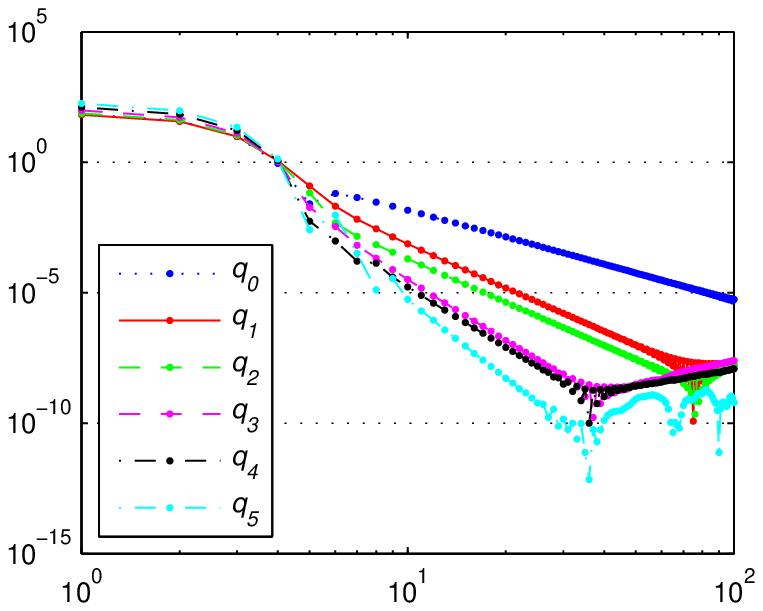}
\caption{Illustration of the decay rate of the numbers $|\beta_n(\pi)|$ (left plot) and $|\gamma_n(\pi)|$ (right plot) for different potentials $q_k$, $k=0,\ldots ,5$ given by \eqref{DecayPotentials2}.}
\label{DecayFig3}
\end{figure}

\subsection{Solution of spectral problems}
One of the possible applications of the proposed representations \eqref{ul via betas} and \eqref{ul prime via gammas} is to the approximate solution of spectral problems. We emphasize that the idea of this subsection is to illustrate the convergence estimates and error bounds from Theorem \ref{Th ul Neumann} and Theorem \ref{Th R2 Legendre} rather than to compete with the best available software packages such as \textsc{Matslise} \cite{LedouxVDaele}.
Our implementation of the approximate method based on the proposed formulas is straightforward. Clearly the method can benefit, e.g.,  from interval subdivision techniques combined with the representation proposed in \cite{KNT 2015}, and we are sure that a robust software package can be created, however we left these tasks for future research.

In all performed numerical experiments the coefficients $\beta_n$ and $\gamma_n$ were computed as was explained at the beginning of this section.  Matlab 2012 in machine precision was used. The optimal number $N$ of terms for the approximations \eqref{ulN} and \eqref{ul prime N} was estimated using the formulas \eqref{Check for precision} for $x=b$, as the value when the partial sums of the series in \eqref{Check for precision} reach the machine-precision induced floor. The upper values like $N=169$ in the following examples appear due to our straightforward implementation of the formula \eqref{beta_n alt}, larger values of $N$ cause the machine precision overflow in computation of $\Gamma(N+l+3/2)$. In all the proposed spectral problems Wolfram Mathematica 8 was able to find the regular solution in the explicit form, which was used to calculate the exact eigenvalues.

\begin{example}\label{Ex SP1}
Consider the following spectral problem
\begin{gather*}
-u''+\left(\frac{l(l+1)}{x^2}+x^2\right) u=\omega^2 u, \quad 0\le x\le \pi,\\
u(\omega, \pi)=0.
\end{gather*}
The value $l=3/2$ was considered in \cite[Example 2]{BoumenirChanane} and \cite[Example 7.3]{CKT2013}. We compared the results with those obtained using \eqref{ulN} with $N=100$. Exact eigenvalues together with the absolute errors of the approximate eigenvalues obtained using different methods are presented in Table \ref{Ex1Table1}. The proposed method is abbreviated as NSBF (from Neumann series of Bessel functions). As one can see from the results, the proposed method is comparable with the SPPS method for lower-index eigenvalues and is clearly superior for the 100th eigenvalue. Additionally it is much faster than the SPPS method.

\begin{table}[htb!]
\centering
\begin{tabular}{cccccc}\hline
$n$ & $\omega_{n}$ (Exact/\textsc{Matslise})  & $\Delta \omega_n$ (NSBF) & $\Delta \omega_n$ (SPPS) & $\Delta \omega_n$ (SLEIGN2) & $\Delta \omega_n$ (\cite{BoumenirChanane})\\\hline
1 &  $2.46294997397397$ & $1.4\cdot 10^{-14}$ & $2.7\cdot 10^{-13}$ & $5.4\cdot 10^{-8}$ & $9.4\cdot 10^{-7}$\\
2 &  $3.28835292994256$ & $5.2\cdot 10^{-14}$ & $6.7\cdot 10^{-12}$ & $1.8\cdot 10^{-7}$ & $1.4\cdot 10^{-5}$\\
3 &  $4.14986421874478$ & $1.2\cdot 10^{-13}$ & $8.2\cdot 10^{-13}$ & $5.0\cdot 10^{-7}$ & $3.1\cdot 10^{-5}$\\
5 &  $6.00758145811600$ & $6.6\cdot 10^{-13}$ & $5.0\cdot 10^{-13}$ & $2.2\cdot 10^{-6}$ & $4.1\cdot 10^{-6}$\\
7 &  $7.93973737689930$ & $2.9\cdot 10^{-13}$ & $6.0\cdot 10^{-13}$ & $7.3\cdot 10^{-6}$\\
10 & $10.8861250916173$ &  $1.5\cdot 10^{-12}$ & $8.6\cdot 10^{-13}$ & $2.4\cdot 10^{-5}$\\
20 & $20.8202301908124$ & $1.4\cdot 10^{-11}$  & $9.6\cdot 10^{-14}$ & $3.4\cdot 10^{-4}$\\
30 & $30.7973502195868$ & $1.5\cdot 10^{-11}$  & $1.9\cdot 10^{-12}$ & $1.6\cdot 10^{-3}$\\
50 & $50.7786768095149$ &  $8.7\cdot 10^{-11}$ & $1.3\cdot 10^{-10}$ & $1.0\cdot 10^{-2}$\\
100 &$100.764442245651$ & $9.4\cdot 10^{-9}$ & $5.3\cdot 10^{-4}$ &\\
\hline
\end{tabular}
\caption{The eigenvalues from Example \ref{Ex SP1} for $l=3/2$ compared to the results reported in \cite{CKT2013}. Since eigenvalues produced by the \textsc{Matslise} package coincide with the exact eigenvalues to all reported digits, we present them in the combined column. $\Delta\omega_n$ denotes the absolute error of the computed eigenvalue $\omega_n$.}
\label{Ex1Table1}
\end{table}

On Figure \ref{Ex1Fig1} we present the absolute errors of the computed eigenvalues for different values of $l$. One can see that the obtained errors follow theoretical predictions given in Proposition \ref{Prop Int l}, Theorem \ref{Th ul Neumann} and Remark \ref{Rmk Accuracy decay}. That is, for integer values of $l$, only few coefficients $\beta_k$ were used (small values of $N$ on the left plot) due to the rapid decrease of them, while for non-integer values of $l$ larger numbers of the coefficients were necessary (large values of $N$ on the right plot). The better accuracy of the first 90-100 eigenvalues on the right plot is due to the estimate \eqref{est uln exp}. The fast deterioration of the accuracy for higher eigenvalues for $l=5$ and $l=10$ is explained in Remark \ref{Rmk Accuracy decay}.

\begin{figure}[htb!]
\centering
\includegraphics[bb=198 309 414 482, width=3in,height=2.4in]
{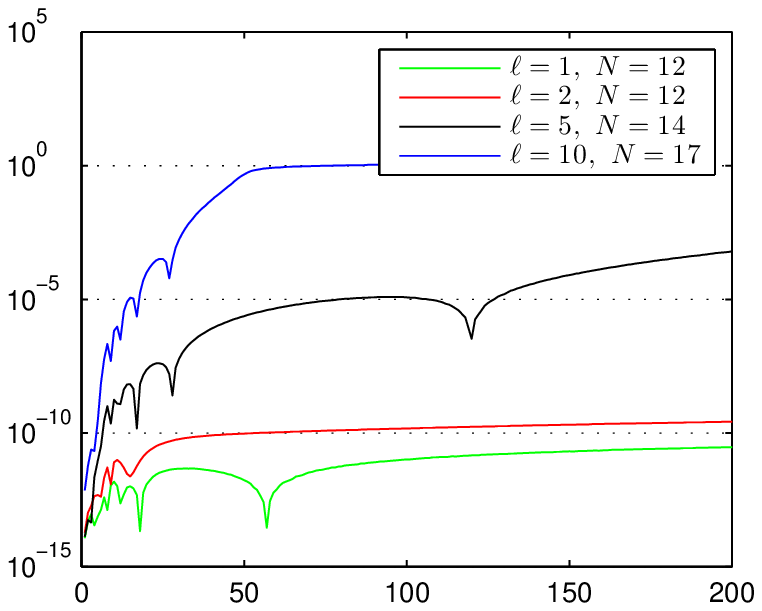}\ \ \
\includegraphics[bb=198 309 414 482, width=3in,height=2.4in]
{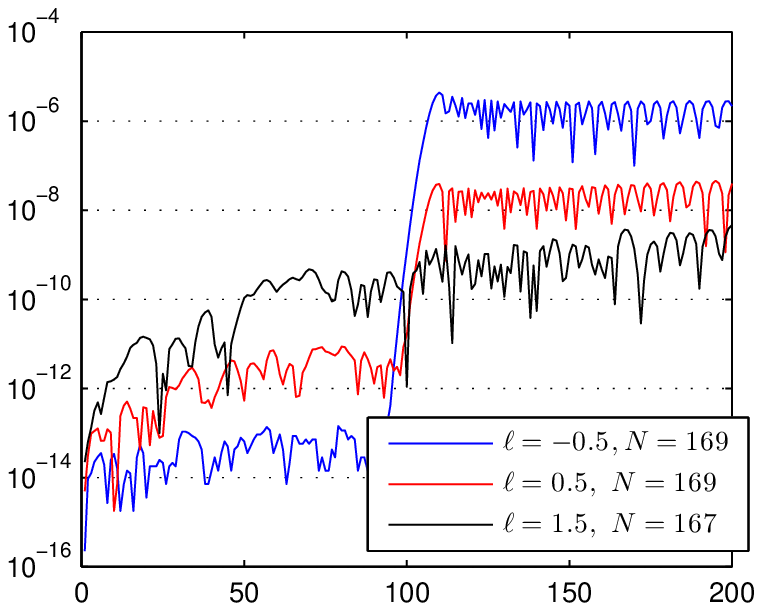}
\caption{Absolute errors of the first 200 eigenvalues for the spectral problem from Example \ref{Ex SP1} for different values of $l$. On the legends the number $N$ used for calculation of the approximate solution \eqref{ulN} is shown next to the value of the parameter $l$.}
\label{Ex1Fig1}
\end{figure}
\end{example}

\begin{example}\label{Ex SP2}
Consider the same equation as in Example \ref{Ex SP1} with a different boundary condition:
\[
u'(\omega, \pi)=0.
\]
Absolute errors of the obtained eigenvalues are presented on Figure \ref{Ex2Fig1}. Again, the results follow the theoretical predictions from Theorem \ref{Th R2 Legendre}.
\begin{figure}[htb!]
\centering
\includegraphics[bb=126 309 486 482, width=5in,height=2.4in]
{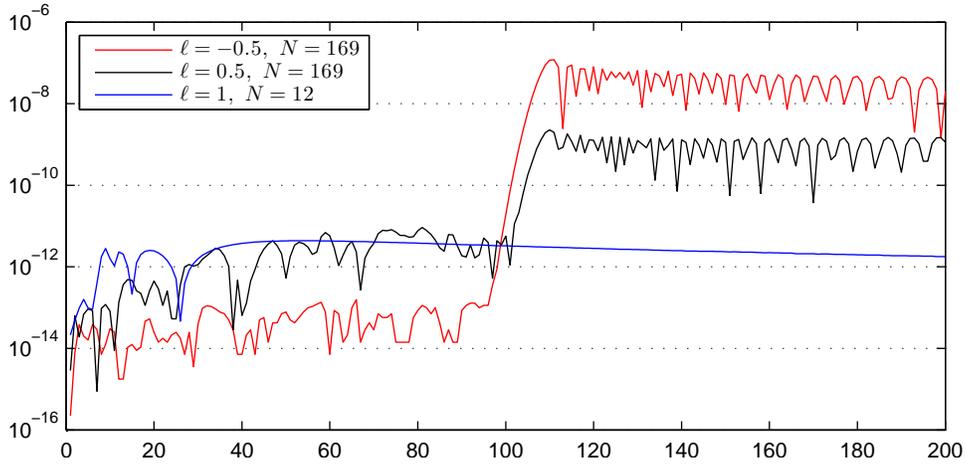}
\caption{Absolute errors of the first 200 eigenvalues for the spectral problem from Example \ref{Ex SP2} for different values of $l$. On the legend the number $N$ used for calculation of the approximate solution \eqref{ulN} is shown next to the value of the parameter $l$.}
\label{Ex2Fig1}
\end{figure}
\end{example}

\begin{example}\label{Ex SP3}
Consider the spectral problem for the hydrogen atom equation \cite[Example 4]{BoumenirChanane}, \cite[Example 7.4]{CKT2013}
\begin{gather*}
-u''+\left(\frac{l(l+1)}{x^2}+\frac 1x\right) u=\omega^2 u, \quad 0\le x\le \pi,\\
u(\omega, \pi)=0.
\end{gather*}
Absolute errors of the obtained eigenvalues are presented on Figure \ref{Ex3Fig1}. The singularity in the potential presents no difficulty for the proposed method. However the case $l=1$, contrary to the previous examples, requires more coefficients $\beta_n$ for the approximate solution \eqref{ulN} to be computed. They do not decay equally fast for integer values of $l$.
\begin{figure}[htb!]
\centering
\includegraphics[bb=126 309 486 482, width=5in,height=2.4in]
{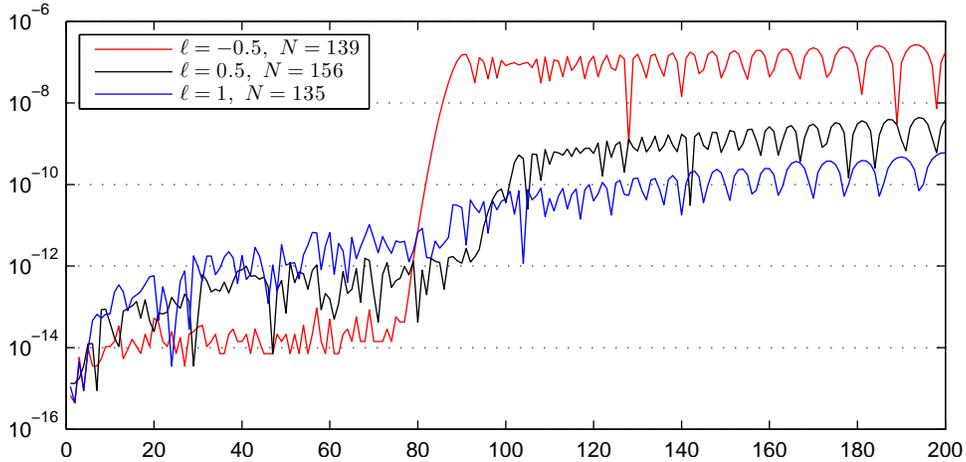}
\caption{Absolute errors of the first 200 eigenvalues for the spectral problem from Example \ref{Ex SP3} for different values of $l$. On the legend the number $N$ used for calculation of the approximate solution \eqref{ulN} is shown next to the value of the parameter $l$.}
\label{Ex3Fig1}
\end{figure}
\end{example}

\end{document}